\documentclass[11pt]{article}
\usepackage{graphicx}
\usepackage{amsfonts}
\usepackage{bbm}
\usepackage{mathrsfs}
\usepackage{mathrsfs,amsmath,amssymb,amsthm}
\usepackage{amssymb}
\usepackage{amsbsy}
\usepackage{color}
\allowdisplaybreaks
\def\ve{\varepsilon}

\theoremstyle{plain}
\theoremstyle{remark}  \newtheorem{remark}{\noindent\mbox{Remark}}
\theoremstyle{remark}  
\theoremstyle{plain}
\theoremstyle{plain}\newtheorem{lemma}{\noindent\mbox{Lemma}}
\theoremstyle{plain} \newtheorem{theorem}{\noindent\mbox{Theorem}}
\theoremstyle{plain}\newtheorem{proposition}{\noindent\mbox{Proposition}}
\theoremstyle{plain}\newtheorem{corollary}{\noindent\mbox{Corollary}}
\theoremstyle{definition} \newtheorem{definition}{\noindent\mbox{Definition}}
\theoremstyle{definition}
 \def\proof{\noindent{\it Proof.~~}}
 \def\qed{\hfill$\Box$\medskip}
 \def\rto{\rightarrow\infty}
 \def\z{\left}
 \def\y{\right}
 \def\no{\nonumber}

\begin{document}
 \title{\textbf{ Local time, upcrossing time and weak cutpoints of  a spatially inhomogeneous random walk on the line}}

\author{   Hua-Ming \uppercase{Wang}$^{\dag,\S}$ and Lingyun \uppercase{Wang}$^\dag$  }
\date{}
\maketitle%
 \footnotetext[2]{School of Mathematics and Statistics, Anhui Normal University, Wuhu 241003, China  }
\footnotetext[4]{Email: hmking@ahnu.edu.cn}
\vspace{-.5cm}

\begin{center}
\begin{minipage}[c]{12cm}
\begin{center}\textbf{Abstract}\quad \end{center}
In this paper, we study a transient spatially inhomogeneous random walk with asymptotically zero drifts on the lattice of the positive half line.
We give criteria for the finiteness of the number of points having exactly the same local time and/or upcrossing time and weak cutpoints (a point $x$ is called a weak cutpoint if the walk never returns to $x-1$ after its first upcrossing from $x$ to $x+1$).
In addition, for the walk with some special local drifts, we also give the order of the expected number of these points in $[1,n].$ Finally, we  show that,  when properly scaled, the number of these points in $[1,n]$ converges in distribution to a random variable with the standard exponential distribution. Our results answer three conjectures related to the local time, the upcrossing time, and the weak cutpoints proposed by E. Cs\'aki, A. F\"oldes,  P. R\'ev\'esz [J. Theoret. Probab. 23 (2) (2010)
624-638].

\vspace{0.2cm}

\vspace{0.2cm}

\end{minipage}
\end{center}

\section{Introduction}\label{s1}

In this paper, we study the local time, upcrossing time and weak cutpoints of a spatially inhomogeneous random walk on $\mathbb Z+ := \{0, 1, 2, ...\}$ with asymptotically zero drifts. To introduce precisely the model,  suppose that $p_k,q_k,k\ge0$ are numbers such that $p_0=1,$ $q_0=0,$ $p_k>0,$ $q_k>0$ and $p_k+q_k=1$ for all $k\ge 1.$
Let $X=\{X_k\}_{k\ge 0}$ be a Markov chain on $\mathbb Z_+$ starting from  $0$ with transition probabilities
\begin{align}
&\mathbb P(X_{k+1}=1|X_k=0)=p_0,\nonumber\\
&\mathbb P(X_{k+1}=n+1|X_k=n)= p_n,\nonumber\\
&\mathbb P(X_{k+1}=n-1|X_k=n)= q_{n}, \nonumber
\end{align}
for $n\ge1$ and $k\ge0.$ Since the transition probabilities are spatially inhomogeneous, we call the chain $X$  {\it a spatially inhomogeneous random walk.} We remark that the chain $X$ is also called a birth and death chain in the literature.

For $k\ge1,$ let $\rho_k=\frac{q_k}{p_k}.$ Then
we have the following criterion of recurrence or transience for the chain $X,$ which can be found, e.g., in Chung's book \cite[ Part I. \S 12]{c}.
\begin{proposition}\label{crt}
  The chain $X$ is transient if and only if $\sum_{k=1}^\infty \rho_1\cdots \rho_k<\infty.$
\end{proposition}
Notice that the local drift of the walk at $n\ge1$ is $2p_n-1.$
What we are concerned about are random walks with asymptotically zero drifts, that is, the case $p_n\rightarrow 1/2$ as $n\rto,$ whose limit behaviors are different in many aspects from those of simple random walks. Such random walks date back to a series of works of Lamperti. In \cite{lam60, lam63}, Lamperti gave the  criteria for recurrence or transience and the existence of the moments. Moreover, in \cite{lam62}, he proved an invariance principle which says that under certain conditions,   $X_{[nt]}/\sqrt{n}$ converges weakly in $D[0,1]$ to a Bessel process. Such an invariance principle was further investigated  and strengthened in Cs\'aki et al. \cite{cfrb}. It is well known that a transient simple random walk grows linearly to infinity. But for a transient spatially inhomogeneous random walk, if the drifts are asymptotically zero, intuitively, the walk grows much slower than the simple random walk. The number of cutpoints on the path of the walk can reflect the speed of the walk.  Roughly speaking, if the walk never returns to $[0,x]$ after its first entry into $[x+1,\infty),$ then $x$ is a cutpoint. By intuition, if the walk runs to infinity more quickly, there are more cutpoints.  But for a spatially inhomogeneous random walk, things are very different. James et al. \cite{jlp} gave an example of a transient random walk which  has only finitely many cutpoints. Such a phenomenon never happens to the simple random walk since it is known that the number of cutpoints of the simple random walk is either 0 (recurrent case) or $\infty$(transient case). Along this line,  Cs\'aki et al. \cite{cfrc}  provided a criterion to determine the finiteness of the number of cutpoints, which we will quote below. We first give the following definition.

\begin{definition}
For $x\in\mathbb Z_+,$ we call $\xi(x):=\sum_{k=0}^{\infty}1_{\{X_{k}=x\}}$  the local time of the chain $X$ at $x$ and  $\xi(x,\uparrow):=\sum_{k=0}^{\infty}1_{\{X_{k}=x,X_{k+1}=x+1\}}$  the upcrossing time of the chain $X$ from $x$ to $x+1,$ respectively. If $\xi(x,\uparrow)=1,$ we call $x$ a cutpoint of $X.$ If $\xi(x)=1,$ we call $x$ a strong cutpoint of $X.$
\end{definition}

For $n\ge m\ge0,$ write
\begin{align}
  D(m,n):=\z\{\begin{array}{ll}
  0,& \text{if } n=m,\\
  1,& \text{if } n=m+1,\\
    1+\sum_{j=1}^{n-m-1}\rho_{m+1}\cdots\rho_{m+j}, &\text{if } n\ge m+2,
  \end{array}\y.\label{ddmn}
\end{align}
and denote
\begin{align}\label{dm}
  D(m):=\lim_{n\rto}D(m,n).
\end{align}
The theorem below gives a criterion for the finiteness of the number of cutpoints.

{\noindent \bf Theorem A}(Cs\'aki et al. \cite{cfrc}) {\it Suppose $p_i\ge 1/2,\ i=1,2,...$ and let $D(n),\ n=1,2,...$ be as in \eqref{dm}.
If $$\sum_{n=2}^\infty\frac{1}{D(n)\log n}<\infty,$$ then the chain $X$ has finitely many cutpoints almost surely. If $D(n)$ is increasing in $n,$ there exist $n_0>0$ and $\delta> 0$ such that  $D(n)\le \delta n\log n$ for all $n\ge n_0$ and $$\sum_{n=2}^\infty\frac{1}{D(n)\log n}=\infty,$$  then  $X$ has infinitely many strong cutpoints almost surely.
}
\vspace{.2cm}
\begin{remark}
  We remark that the monotonicity condition of $D(n)$ is not contained in \cite[Theorem 1.1]{cfrc}. But we knew from Professor A. F\"oldes, one author of \cite{cfrc}, that  such a  condition is indeed required  for proving the last display on page 634 of \cite{cfrc}.
\end{remark}

 Cs\'aki et al. proposed  in \cite{cfrc}  four conjectures related to the weak cutpoints, local time and upcrossing time of the walk.  We  quote here the first three of them word for word.

\noindent{\bf Open Problems(Cs\'aki et al. \cite{cfrc})}
\begin{itemize}

\item[1.] It would be interesting to know whether Theorem A also holds for the number
of sites with $\xi(R) = a$ or $\xi(R,\uparrow) = a$ for any fixed integer $a > 1$, i.e., whether
we have the same criteria for $\{\xi(R) = a\}$ and $\{\xi(R,\uparrow) = a\}$ to occur infinitely
often almost surely for any positive integer $a.$

  \item[2.] Call the site $R$ a {\it weak cutpoint} if, for some $k$, we have $X_k=R,$ $X_i \le R,\ i=0,1,...,k-1,$ and $X_i\ge R,\ i=k+1,k+2,....$ One would like to know whether Theorem A can be extended for the number of weak cutpoints.

\item[3.] It would be interesting to know whether Theorem A holds for cutpoints with
a given local time, i.e., for $\{\xi(R) = a, \xi(R,\uparrow) = 1\}$ or, in general, $\{\xi(R) = a,
\xi(R,\uparrow) = b\}$ infinitely often almost surely, with positive integers $a, b.$
\end{itemize}

 The main task of the paper is to answer the above open problems.
\begin{definition}\label{def2}
Let $A$ and $B$ be two subsets of $\{1,2,...\}.$
 We denote by
\begin{align*}
  C(A,B):=\{x\in \mathbb Z_+: \xi(x)\in A,\xi(x,\uparrow)\in B\}
\end{align*}
 the collection of nonnegative sites  at which the local times and the upcrossing times of the  walk belong to $A$ and $B,$  respectively. Moreover, consider the weak cutpoint defined in the above open problem 2.
We denote by \begin{align*}
  C_w:=\{x\in \mathbb Z_+: x \text{ is a weak cutpoint}\}
\end{align*} the collection of all weak cutpoints.
\end{definition}
In what follows,  if $A=\{a\},$ instead of $C(\{a\},B),$ we write simply $C(a,B).$ The notations $C(A,a)$ and $C(a,b)$ can be understood similarly.
 Also, for simplicity, we write
\begin{align}
  C(*,a)=C(\{a,a+1,...\},a) \text{ and } C(a,*)=C(a,\{1,...,a\}).\label{cstar}
\end{align}
Notice that $C(*,a)$ is the collection of sites at which the upcrossing times of the walk are exactly $a,$ while $C(a,*)$ is the collection of sites at which the local times of the walk are exactly $a.$ Finally, we use $|C|$ to denote the cardinality for a set $C.$

The theorem below provides criteria to tell whether $C(A,a),$ $C(a,B)$ or $C_w$ are finite or not, answering the above-mentioned open problems.

\begin{theorem}\label{main}  Let $D(n),\ n=1,2,...$ be as in \eqref{dm}.
Suppose that $\rho_k$ is increasing in $k>N_0$ for some $N_0>0$ and $\rho_k\rightarrow 1$ as $k\rto.$ If $$\sum_{n=2}^\infty\frac{1}{D(n)\log n}<\infty,$$ then almost surely, we have \begin{enumerate}
  \item[{\rm(i)}] $|C(A,a)|<\infty$ for each $a\in \{1,2,...\}$ and $A\subseteq \{a,a+1,...\};$
  \item[{\rm(ii)}] $|C(a,B)|<\infty$ for each $a\in \{1,2,...\}$ and  $B\subseteq\{1,2,...,a\};$
       \item[{\rm(iii)}] $|C_w|<\infty.$
      \end{enumerate}
If $D(n)\le \delta n\log n,$ $n\ge n_0$ for some $n_0>0$ and $\delta>0$ and  $$\sum_{n=2}^\infty\frac{1}{D(n)\log n}=\infty,$$ then almost surely, we have
\begin{itemize}
  \item[{\rm(i)}] $|C(A,a)|=\infty$ for each $a\in \{1,2,...\}$ and $\phi\ne A\subseteq \{a,a+1,...\};$
  \item[{\rm(ii)}] $|C(a,B)|=\infty$ for each $a\in \{1,2,...\}$ and  $\phi\ne B\subseteq\{1,2,...,a\};$
       \item[{\rm(iii)}] $|C_w|=\infty.$
 \end{itemize}
\end{theorem}
\begin{remark} We explain how Theorem 1 answers the above open problems. Noticing that $C(a,\{1,2,...,a\})=\{R\in \mathbb Z_+: \xi(R)=a\}$ and $C(\{a,a+1,...\},a)=\{R\in \mathbb Z_+: \xi(R,\uparrow)=a\},$ thus
the criteria for $|C(a,\{1,2,...,a\})|<\infty$ and $|C(\{a,a+1,...\},a)|<\infty$ answer the above open problem 1.  Clearly, the criterion for $|C_w|<\infty$ gives an answer to the above open problem 2 and the criterion for the finiteness of $C(a,b)$ provides an answer to the above open problem 3.

 Fix $a\in \{1,2,...\},$ $A\subseteq \{a,a+1,...\}$ and   $B\subseteq\{1,2,...,a\}.$ If $D(n)=\infty$ for some and hence all $n\ge0,$ then it follows from Proposition \ref{crt} that the chain $X$ is recurrent. Thus, for any site $x\ge 0,$ we have $\xi(x)=\infty.$ As a consequence, we get  $|C(A,a)|=|C(a,B)|=|C_w|=0,$ which coincides with the convergent part of Theorem \ref{main} since
  in this case, we always have $\sum_{n=2}^\infty\frac{1}{D(n)\log n}=0.$  The interesting phenomenon arises when $D(n)<\infty$ for all $n$ and $\sum_{n=2}^\infty\frac{1}{D(n)\log n}<\infty.$ In this case, the walk is transient,  but almost surely, we have  $|C(A,a)|+|C(a,B)|+|C_w|<\infty.$

Finally, we mention that $C(*,1)$ and $C(1,1)$ are  collections of the cutpoints and strong cutpoints, respectively. From this point of view,  Theorem \ref{main} is a generalization of Theorem A.

%(iii) To prove Theorem \ref{main}, we follow the ideas of \cite{cfrc,jlp}. The main difficulty is to compute the joint probability of the events $\{x\in C(a,A)\}$ and $\{x+y\in C(a,A)\}$ and prove further that they are asymptotically independent as $y\rto$,  see Lemma \ref{aaxy} and Proposition \ref{pxy} below.
\end{remark}

As pointed out in Cs\'aki et al. \cite{cfrc}, the condition ``$D(n)\le \delta n\log n,$ $n\ge n_0$ for some $n_0>0$ and $\delta>0$" is a technical one. We now give a concrete example for which such a condition is satisfied naturally.

Fix $\beta\in \mathbb R.$ Clearly, there exists $n_1>0$ such that $\frac{1}{4}\z(\frac{1}{n}+\frac{1}{n(\log\log n)^\beta}\y)\in (0,1/2)$ for all $n>n_1.$  For $n\ge1,$ set \begin{align}\label{dr}
  r_n=\z\{\begin{array}{ll}
            \frac{1}{4}\z(\frac{1}{n}+\frac{1}{n(\log\log n)^\beta}\y), & \text{if }n\ge n_1,  \\
           r_{n_1},  &  \text{if }n< n_1,
          \end{array}
  \right.
\end{align} which serves as perturbation that will be added to the transition probability of the simple recurrent random walk.
 We emphasize that the perturbation $r_n$ in \eqref{dr} is taken from \cite{cfrc} without change.
\begin{lemma}\label{dnl} For $n\ge1,$ let $r_n$ be the one in \eqref{dr} and set $p_n=\frac{1}{2}+r_n.$ Then
$
 D(n)\sim n(\log\log n)^\beta \text{ as }n\rto.
$
\end{lemma}
The proof of  Lemma \ref{dnl} is given  in Section \ref{s5}.
It is easy to see from Lemma \ref{dnl} that there is $n_0>0$ such that $D(n)<n\log n,$ $n\ge n_0.$ Thus, applying Theorem \ref{main}, we have the following corollary, which gives sharp criteria for the finiteness of the sets $C(A,a)$, $C(a,B)$ and $C_w.$
\begin{corollary}\label{cor1} Fix $a\in \{1,2,...\},$ $\phi\ne A\subseteq \{a,a+1,...\}$ and   $\phi\ne B\subseteq\{1,2,...,a\}.$ For $n\ge1,$ let $r_n$ be the one in \eqref{dr} and set $p_n=\frac{1}{2}+r_n.$ If $\beta>1,$ we have $|C(A,a)|+|C(a,B)|+|C_w|<\infty$ almost surely. If  $\beta\le 1,$ we have $|C(A,a)|=|C(a,B)|=|C_w|=\infty$ almost surely.
\end{corollary}

Another interesting question is to estimate the cardinality of the set $C\cap [1,n]$ for $C=C_w,$ $C(A,a)$ or $C(a,B)$  and $n$ large enough.
We first compute the expectations of the cardinalities of these sets.
\begin{proposition}\label{ec} Suppose  $a\ge b\ge1$ are  positive integers.  For $n\ge1,$ let $r_n$ be as in \eqref{dr} and set $p_n=\frac{1}{2}+r_n.$ Then we have
  \begin{align}\label{ecab}
  \lim_{n\rto}\frac{(\log\log n)^{\beta}}{\log n}\mathbb E|C(a,b)\cap[1,n]|=\frac{1}{2^a}\binom{a-1}{b-1}
\end{align}
 and
 \begin{align}\label{ecw}
   \lim_{n\rto}\frac{(\log\log n)^{\beta}}{\log n}\mathbb E|C_w\cap[1,n]|=2.
 \end{align}
  \end{proposition}
  \begin{remark}
  Fix $a\in \{1,2,...\},$ $A\subseteq \{a,a+1,...\}$ and   $B\subseteq\{1,2,...,a\}.$ By the linear property of the expectation, it follows immediately from \eqref{ecab} that
  \begin{align}\label{ecaa}
  &\lim_{n\rto}\frac{(\log\log n)^{\beta}}{\log n}\mathbb E|C(A,a)\cap[1,n]|=\sum_{x\in A}\frac{1}{2^{x}}\binom{x-1}{a-1},\\
  &\lim_{n\rto}\frac{(\log\log n)^{\beta}}{\log n}\mathbb E|C(a,B)\cap[1,n]|=\frac{1}{2^{a}}\sum_{x\in B}\binom{a-1}{x-1}.\label{eca}
\end{align}
     Recall that $C(a,*)$ is the collection of sites at which the local times of the walk are exactly $a,$ while $C(*,a)$  is the collection of sites at which the upcrossing times of the walk are exactly $a.$
   From \eqref{ecaa} and\eqref{eca}, we get
    \begin{align}
      &\lim_{n\rto}\frac{(\log\log n)^{\beta}}{\log n}\mathbb E|C(*,a)\cap[1,n]|=1,\no\\
      &\lim_{n\rto}\frac{(\log\log n)^{\beta}}{\log n}\mathbb E|C(a,*)\cap[1,n]|=1/2.\no
    \end{align}
  \end{remark}

Inspired by Proposition \ref{ec},  one may expect that
$\frac{(\log\log n)^\beta}{\log n}\z|C(a,b)\cap [1,n]\y|\rightarrow \frac{1}{2^a}\binom{a-1}{b-1}$ for $a\ge b\ge1$  or $\frac{(\log\log n)^\beta}{\log n}\z|C_w\cap [1,n]\y|\rightarrow 2$ almost surely as $n\rto.$ However, this is not the case.  The main reason is that the events $\{x\in C(a,b)\}(\text{or } x\in C_w), x=1,2,...,n$ are not independent. We have the following theorem.
\begin{theorem}\label{cs} Suppose $a\ge b\ge1$ are positive integers.  Let $r_1=1/4 $ and  $r_n=\frac{1}{2n}$ for $n\ge2.$  Set $p_n=\frac{1}{2}+r_n$ for $n\ge1.$
  Then
  \begin{align*}
  \frac{\z|C(a,b)\cap [1,n]\y|}{2^{-a}\binom{a-1}{b-1}\log n}\overset{ d}{\rightarrow} S,\ \frac{|C(*,a)\cap[1,n]|}{\log n}\overset{ d}{\rightarrow} S\text{ and }  \frac{\z|C_w\cap [1,n]\y|}{2\log n}\overset{ d}{\rightarrow} S
  \end{align*}  as $n\rto,$ where (and in what follows) the notation $\overset{ d}{\rightarrow}$ means the convergence in distribution and $S$ is an exponential random variable with $\mathbb P(S>t)=e^{-t}, t>0.$
\end{theorem}

 \begin{remark} Theorem \ref{cs} is proved by the standard moment method. The key step is to find the limits of  the moments of $\frac{2^a\z|C(a,b)\cap [1,n]\y|}{\binom{a-1}{b-1}\log n},$ $\frac{\z|C(*,a)\cap [1,n]\y|}{\log n},$ or $\frac{\z|C_w\cap [1,n]\y|}{2\log n},$ which depend on the joint probability of $\{j_1\in C,...,j_m\in C\}$ for $C=C(a,b)$, $C(*,a)$ or $C_w,$ $0< j_1<j_2<...<j_m.$ Except for the special case $a\ge b=1,$ it is hard to give explicitly the formulae of such joint probabilities. But we can show that  the events $\{j_i\in C\},$  $i=1,...,m$ satisfy certain Markov properties, that is, $\mathbb P(j_{s+1}\in C|j_s\in C,...,j_1\in C)=\mathbb P(j_{s+1}\in C|j_s\in C)$ for $1\le s<m,$ $C=C(a,b)$, $C(*,a)$ or $C_w.$
 So it suffices to compute and estimate the two-dimensional joint probabilities of these events.
 Fix an integer $a>1.$ For general $A\subseteq \{a,a+1,...\}$ and  $B\subseteq\{1,2,...,a\},$ we currently do not know how to get the limit distributions of $|C(A,a)\cap [1,n]|$ and $|C(a,B)\cap [1,n]|.$ For example, for $a>1,$ $C(a,*)\cap [1,n]$ is the collection of sites in $[1,n]$ at which the local times of the walk are exactly $a.$ But neither is  $\{\xi(x)\}_{x\ge 0}$  a Markov chain nor is it possible for us to compute and estimate the joint probability of $\{\xi(j_1)=a,...,\xi(j_m)=a\}.$ Therefore, by our method, we can not find the limit distribution of $|C(a,*)\cap [1,n]|.$
\end{remark}

\vspace{.4cm}

\noindent{\bf Outline of the paper.} The present paper is built up as follows. Firstly,  besides introducing  the exit probabilities of the walk, we  prove some combinatorial results in Section \ref{ar}, which play key roles in proving the main results.  Then, by summarizing the proofs of Cs\'aki et al. \cite[Theorem 1.1]{cfrc} and James et al. \cite[Theorem 3.1]{jlp}, we get a version of the Borel-Cantelli lemma in Section \ref{sec3},  which will be used to prove Theorem \ref{main}. In Section \ref{sec4}, we express the upcrossing times in terms of loops, moves and escapes.
The probability of $\{x\in C\},$ and the joint probability and dependence of $\{x\in C\}$ and $\{y\in C\}$ for $C=C_w,$ $C(*,a)$ or $C(a,b)$ are studied in Section \ref{sec5}. Section \ref{s4} is devoted to proving Theorem \ref{main}. Finally, Proposition \ref{ec} and Theorem \ref{cs} are proved in Section \ref{s5} and Section \ref{s6}, respectively.

\vspace{.3cm}

{\it\noindent Notations:}
When writing $A(n)\sim B(n)$, we mean  $A(n)/B(n)\rightarrow 1$ as $n \rightarrow \infty$. The notation $A(n)=O(B(n))$ means there is a constant $c>0$ such that $\lvert A(n) \rvert< cB(n) $ for all $n$ large enough.
We also adopt the convention that an empty product equals identity and an empty sum equals $0.$ Finally, throughout the paper, $0<c <\infty$ is assumed to be a constant whose value may change from line to line.

\section{Auxiliary results}\label{ar}
In this section, we introduce first the exit probabilities of the walk from certain intervals and then give some combinatorial results that are  helpful for the proofs of both Theorem \ref{main} and Theorem \ref{cs}.
\subsection{Exit probabilities}
For $n\ge k\ge m\ge 0$ set
\begin{align}
  P_k(m,n,-)&:=\mathbb P(X \text{ hits } m \text{ before it hits }n |X_0=k),\label{pmnm}\\
  P_k(m,n,+)&:=\mathbb P(X \text{ hits } n \text{ before it hits }m |X_0=k),\label{pmnp}
\end{align}
which are the exit probabilities of the walk from $[m,n].$
Evidently, we have $$P_k(m,n,-)+P_k(m,n,+)=1, \text{ for }n\ge k\ge m\ge 0.$$
Set also
$$P_0(-1,n,+)=\mathbb P(X \text{ hits } n \text{ before it hits }-1 |X_0=0).$$
Since $q_0=0,$ we must have
\begin{align}P_0(-1,n,+)=1 \text{ for }n\ge0.\no\end{align}

The following lemma is from  Cs\'aki et al. \cite{cfra}, see Lemma 2.2 therein, which  provides  formulae for the exit probabilities.
\begin{lemma}\label{ex}
For $n\ge k\ge m\ge0,$ we have
  \begin{align}
    P_k(m,n,-)=1-\frac{D(m,k)}{D(m,n)}.\no
  \end{align}
\end{lemma}
%The proof of Lemma \ref{ex} is very standard. We refer the reader to \cite[Lemma 2.2]{cfra}.
It is easily seen from Lemma \ref{ex} that for $0\le x<y,$
\begin{align}\label{ep}
  P_{x+1}(x,y,+)=\frac{1}{D(x,y)},\  P_{x+1}(x,\infty,+)=\frac{1}{D(x)},\ P_{y-1}(x,y,+)=\frac{D(x,y-1)}{D(x,y)}.
\end{align}
Notice that these exit or escape probabilities are written in terms of
 $D(m)$ and $D(m,n)$. By some direct computation, we see that
 \begin{gather}
   D(m)=1+\rho_{m+1}D(m+1),m\ge0,\no\\
    D(m,n)=1+\rho_{m+1}D(m+1,n), n>m\ge0,\no\\
    D(m,n)=D(m,n-1)+\rho_{m+1}\cdots \rho_{n-1}, n> m\ge0.\no
 \end{gather}
 Furthermore, it is shown in Cs\'aki et al. \cite{cfrc} (see display (2.3) therein) that
\begin{align}\label{dmn}
  \frac{D(m,n)}{D(m)}=1-\prod_{i=m}^{n-1}\z(1-\frac{1}{D(i)}\y),n>m\ge 0.
\end{align}
Finally, we give the following estimations.
\begin{lemma}\label{edxy}
  Suppose that $\rho_k\rightarrow 1$ as $k\rto.$ Then, for each $\varepsilon>0$ there exists  $N>0,$ which is independent of $x,$  such that for $x>N,\  y>x+N,$ we have
  \begin{gather}
    \frac{1}{D(x)}\le \varepsilon,\ \frac{1}{D(x,y)}\le \varepsilon,\   1-\ve\le \frac{D(x-1)}{D(x)}\le 1+\ve, \no\\
    1-\ve\le\frac{D(x-1,y)}{D(x,y)}\le 1+\ve,\ 1\le \frac{D(x,y)}{D(x,y-1)}\le 1+\varepsilon.\no
  \end{gather}

\end{lemma}
\begin{proof}Fix $0<\eta<1.$  Choose  $M_1>0,$ which is independent of $x,$ such that $(1-\eta)^{y-x-1}<\eta$ for all $y-x\ge M_1.$ Since $\lim_{k\rto}\rho_k=1,$ there is a number $M_2>0$ such that for $k\ge M_2,$ we have $1+\eta>\rho_k^{-1}\vee\rho_k\ge \rho_k^{-1}\wedge\rho_k>1-\eta.$ Let $N=M_1\vee M_2.$ Then, we have
\begin{align}
  D(x,y)&=1+\sum_{k=x+1}^{y-1}\rho_{x+1}\cdots\rho_{k}>
  1+\sum_{k=1}^{y-x-1}(1-\eta)^k\no\\
  &=\frac{1-(1-\eta)^{y-x}}{\eta}\ge \frac{1-\eta}{\eta}, \ x>N,\ y>x + N\no
\end{align}
 and
\begin{align}
  1&\le\frac{D(x,y)}{D(x,y-1)}= 1+\frac{\rho_{x+1}\cdots\rho_{y-1}}{1+\sum_{k=x+1}^{y-2}\rho_{x+1}\cdots \rho_{k}}\no\\
  &=1+\frac{1}{\sum_{k=x+1}^{y-1}\rho_{k}^{-1}\cdots \rho_{y-1}^{-1}}\le1+\frac{1}{\sum_{k=1}^{y-x-1}(1-\eta)^k}\no\\
  &=1+\frac{\eta}{(1-\eta)(1-(1-\eta)^{y-x-1})}\no\\
  &\le 1+\frac{\eta}{(1-\eta)^2}, \ x>N,\ y>x+N.\no
\end{align}
Moreover, since
$
  \frac{D(x-1,y)}{D(x,y)}=\rho_x+\frac{1}{D(x,y)},
$
we have
\begin{align*}
  1-\eta<\frac{D(x-1,y)}{D(x,y)}<1+\eta+\frac{\eta}{1-\eta}, \ x>N,\ y>x+N.
\end{align*}
 Therefore, fixing $\ve>0,$ the lemma  is proved by choosing $\eta$ small enough. \qed
\end{proof}

\subsection{Some combinatorial results}\label{s2}
 To begin with, we study the cardinality of certain simplexes.
For positive integers $a\ge j\ge i\ge 1,$ set
\begin{align}\label{saj}
  &S(a,j)=\z\{(v_1,...,v_j):\ v_k\in \mathbb Z_+/\{0\},\ j\ge k\ge1, \ \sum_{k=1}^jv_k=a\y\},\\
  &\tilde{S}(a,j)=\z\{(v_1,...,v_j):\  v_k\in \mathbb Z_+,\ j\ge k\ge1,\ v_j\ge 1, \  \sum_{k=1}^j v_k=a\y\},\no\\
  &\tilde{S}_i(a,j)=\z\{(v_1,...,v_j)\in \tilde{S}(a,j):\  \sum_{k=1}^j1_{v_k\ne0}=i\y\}.\no
\end{align}

%For simplicity, in what follows we denote by $\vec{v}=(v_1,\cdots,v_j)$ a $j$-dimensional vector.

 \begin{lemma}\label{ns} For $a\ge j\ge i\ge1,$ we have
 $|S(a,j)|=\binom{a-1}{j-1}$ and $|\tilde{S}_i(a,j)|=\binom{j-1}{i-1}\binom{a-1}{i-1}.$
    \end{lemma}
\proof The first assertion is obvious. Indeed, arrange $a$
balls in a lineup. This lineup has $a-1$ spaces among the balls. Choose $j-1$ of those spaces and erect
a wall at each of the selected places to get an arrangement of $j$ nonempty sets of balls. Obviously,
this can be done by $\binom{a-1}{j-1}$
ways. Thus $|S(a,j)|=\binom{a-1}{j-1}.$

Next, we compute the cardinality of $\tilde{S}_i(a,j).$
 Consider a vector $(v_1,...,v_j)\in \tilde{S}_i(a,j).$ There are exactly $i$ strictly positive coordinates. Since by definition, $v_j\ge1,$ thus for some $k_1,k_2,...,k_{i-1},$ we have $v_{k_l}> 0,l=1,...,i-1$ and $v_{k_1}+\cdots+v_{k_{i-1}}+v_j=a.$ But there are $\binom{j-1}{i-1}$ ways to choose $i-1$ coordinates from $j-1$ coordinates. Therefore, we must have $$|\tilde{S}_i(a,j)|=\binom{j-1}{i-1}|S(a,i)|=\binom{j-1}{i-1}\binom{a-1}{i-1}.$$
  The lemma is proved. \qed

The next lemma is crucial for proving Theorem \ref{cs}.
\begin{lemma}\label{lems}Write $j_0=0$ and fix $k\ge 1, n_2\ge1.$ We have
\begin{align}\label{ls}
  \lim_{n\rto}\frac{1}{(\log n)^k}\sum_{\begin{subarray}{c}
    0< j_1<...<j_k\le n\\
    j_i-j_{i-1}\ge n_2, i=1,...,k
  \end{subarray}}\frac{1}{j_1(j_2-j_1)\cdots(j_k-j_{k-1})}=1.
\end{align}
  \end{lemma}
  \begin{remark} We have never seen formula \eqref{ls} in the literature we are aware of, so we do not know whether it is actually new or not. If that is the case, such a formula may have an independent interest in mathematical analysis.
  \end{remark}

\proof
To start with, we prove \eqref{ls} for $n_2=1.$
If $k=1,$ it is clear that $$\lim_{n\rto}\frac{1}{\log n}\sum_{j_1=1}^n \frac{1}{j_1}=1.$$  Next, we fix $k\ge2.$
 Notice  that
\begin{align}\no
  \frac{1}{j_1(j_2-j_1)\cdots(j_k-j_{k-1})}=\prod_{i=1}^{k-1}\z(\frac{1}{j_i}+\frac{1}{j_{i+1}-j_i}\y)\frac{1}{j_k}.
\end{align}
We thus have
\begin{align}\label{ps}
  &\sum_{0< j_1<...<j_k\le n}\frac{1}{j_1(j_2-j_1)\cdots(j_k-j_{k-1})}\no\\
  &\quad\quad=\sum_{j_k=k}^{n}\frac{1}{j_k}\sum_{j_{k-1}=k-1}^{j_k-1}\z(\frac{1}{j_{k-1}}+\frac{1}{j_k-j_{k-1}}\y)\cdots \sum_{j_1=1}^{j_2-1}\z(\frac{1}{j_{1}}+\frac{1}{j_2-j_{1}}\y).
  \end{align}
We claim that for $1\le i\le k-1,$ \begin{align}\label{rki}
  \sum_{j_i=i}^{j_{i+1}-1}&\z(\frac{1}{j_{i}}+\frac{1}{j_{i+1}-j_{i}}\y)i(i-1+\log j_i)^{i-1}
  \le (i+1)(i+\log j_{i+1})^i,\\
  \sum_{j_i=i}^{j_{i+1}-1}&\z(\frac{1}{j_{i}}+\frac{1}{j_{i+1}-j_{i}}\y)i(\log j_i-\log(i-1))^{i-1}\ge (i+1)(\log j_{i+1}-\log i)^i.\label{rkl}
\end{align}
Indeed, \eqref{rki} holds trivially for $i=1.$ For $2\le i\le k-1,$ since
\begin{align}
  \sum_{j_i=i}^{j_{i+1}-1}&\frac{ i(i-1+\log j_i)^{i-1}}{j_{i}}\le i\int_{i-1}^{j_{i+1}-1}\frac{(i-1+\log(x+1))^{i-1}}{x}dx\no\\
  &\le i\int_{i-1}^{j_{i+1}-1}\frac{(i+\log x)^{i-1}}{x}dx\le(i+\log j_{i+1})^i\no
\end{align}
and
\begin{align}
  \sum_{j_i=i}^{j_{i+1}-1}&\frac{i(i-1+\log j_i)^{i-1}}{j_{i+1}-j_{i}}=\sum_{j_i=1}^{j_{i+1}-i}\frac{i(i-1+\log (j_{i+1}-j_{i}))^{i-1}}{j_{i}}\no\\
  &\le \sum_{j_i=1}^{\z[j_{i+1}/2\y]}+\sum_{j_i=1+\z[j_{i+1}/2\y]}^{j_{i+1}-i}\frac{i(i-1+\log (j_{i+1}-j_{i}))^{i-1}}{j_{i}}\no\\
  &\le i(i+\log j_{i+1})^{i-1}\sum_{j_i=1}^{\z[j_{i+1}/2\y]}\frac{1}{j_i}
  +i(i+\log(j_{i+1}/2))^{i-1}\sum_{j_i=1+\z[j_{i+1}/2\y]}^{j_{i+1}-i}\frac{1}{j_i}\no\\
  &\le i(i+\log(j_{i+1}/2))^{i-1}(1+\log(j_{i+1}/2))+i\log 2(i+\log(j_{i+1}/2))^{i-1}\no\\
  &\le i(i+\log j_{i+1})^{i},\no
\end{align} then \eqref{rki} is true. We now proceed to prove \eqref{rkl}. To this end, note that
\begin{align}
  \sum_{j_i=i}^{j_{i+1}-1}&\frac{ i(\log j_i-\log(i-1))^{i-1}}{j_{i}}\ge i\int_{i}^{j_{i+1}}\frac{(\log(x-1)-\log(i-1))^{i-1}}{x}dx\no\\
  &\ge i\int_{i}^{j_{i+1}}\frac{(\log x-\log i)^{i-1}}{x}dx=(\log j_{i+1}-\log i)^i\no
\end{align}
and
\begin{align}
  \sum_{j_i=i}^{j_{i+1}-1}&\frac{i(\log j_i-\log(i-1))^{i-1}}{j_{i+1}-j_{i}}=\sum_{j_i=1}^{j_{i+1}-i}\frac{i(\log (j_{i+1}-j_i)-\log(i-1))^{i-1}}{j_{i}}\no\\
  &\ge \sum_{j_i=1}^{\z[j_{i+1}/i\y]}\frac{i(\log (j_{i+1}-j_i)-\log(i-1))^{i-1}}{j_{i}}\no\\
  &\ge i\z(\log\frac{(i-1)j_{i+1}}{i} -\log(i-1) \y)^{i-1}\sum_{j_i=1}^{\z[j_{i+1}/i\y]}\frac{1}{j_i}\no\\
  &\ge i(\log j_{i+1}-\log i)^{i-1}\log([j_{i+1}/i]+1)\no\\
  &\ge i(\log j_{i+1}-\log i)^{i}.\no
\end{align}
We thus get \eqref{rkl}.

Using \eqref{rki} and \eqref{rkl} time and again, from \eqref{ps}, we get
\begin{align}\label{ub}
  &\sum_{0< j_1<...<j_k\le n}\frac{1}{j_1(j_2-j_1)\cdots(j_k-j_{k-1})}\no\\
  &\quad\le \sum_{j_k=k}^{n}\frac{1}{j_k}k(k-1+\log j_{k})^{k-1}\le\int_{k-1}^n \frac{k(k-1+\log (x+1))^{k-1} }{x}dx  \no\\
    &\quad\le\int_{k-1}^n \frac{k(k+\log x)^{k-1} }{x}dx\le (k+\log n)^k
  \end{align}
  and
  \begin{align}\label{lbs}
  &\sum_{0< j_1<...<j_k\le n}\frac{1}{j_1(j_2-j_1)\cdots(j_k-j_{k-1})}\no\\
  &\quad\ge \sum_{j_k=k}^{n}\frac{1}{j_k}k(\log j_{k}-\log(k-1))^{k-1}\ge\int_{k}^{n+1} \frac{k(\log (x-1)-\log(k-1))^{k-1} }{x}dx  \no\\
    &\quad\ge\int_{k}^n \frac{k(\log x-\log k)^{k-1} }{x}dx= (\log n-\log k)^k.
  \end{align}
  Dividing  \eqref{ub} and \eqref{lbs} by $(\log n)^k$ and taking the upper and lower limits, we obtain
\begin{align}\label{pj1}
  \lim_{n\rto}\frac{1}{(\log n)^k}\sum_{0< j_1<...<j_k\le n}\frac{1}{j_1(j_2-j_1)\cdots(j_k-j_{k-1})}= 1.
\end{align}
    Applying \eqref{pj1}, for any $ l_0\ge 1$ and $1\le i\le k,$ we have
    \begin{align}
  \lim_{n\rto}&\frac{1}{(\log n)^k}\sum_{\begin{subarray}{c}0< j_1<...<j_k\le n,\\
  j_{i}-j_{i-1}=l_0\end{subarray}}\frac{1}{j_1(j_2-j_1)\cdots(j_k-j_{k-1})}\no\\
  &=\frac{1}{l_0}\lim_{n\rto}\frac{(\log(n-l_0))^{k-1}}{(\log n)^{k}}\frac{1}{(\log (n-l_0))^{k-1}}\no\\
  &\quad\quad\quad\quad\times\sum_{0< j_1<...<j_{k-1}\le n-l_0}\frac{1}{j_1(j_2-j_1)\cdots(j_{k-1}-j_{k-2})}\no\\
  &=0.\no
\end{align}
Consequently,  we infer that for any $1\le i\le k,$
\begin{align}\label{lsb}
  \lim_{n\rto}&\frac{1}{(\log n)^k}\sum_{\begin{subarray}{c}0< j_1<...<j_k\le n,\\
  j_{i}-j_{i-1}\le n_2\end{subarray}}\frac{1}{j_1(j_2-j_1)\cdots(j_k-j_{k-1})}=0.
\end{align}
Putting \eqref{pj1} and \eqref{lsb} together, we get \eqref{ls}. Lemma \ref{lems} is proved. \qed

%
%We thus finish the proof of \eqref{ls} when $l=1.$
%Assume now \eqref{ls} holds for $l=1,...,m.$ Then
%\begin{align}
%\lim_{n\rto}&\frac{1}{(\log n)^k}\sum_{m+1\le j_1<...<j_k\le n}\frac{1}{j_1(j_2-j_1)\cdots(j_k-j_{k-1})}\no\\
%  &=\lim_{n\rto}\frac{1}{(\log n)^k}\sum_{m\le j_1<...<j_k\le n}\frac{1}{j_1(j_2-j_1)\cdots(j_k-j_{k-1})}\no\\
%  &\quad\quad-\lim_{n\rto}\frac{1}{(\log n)^k}\sum_{m+1\le j_2<\cdots <j_k\le n}\frac{1}{m(j_2-m)(j_3-j_2)\cdots (j_k-j_{k-1})}\no\\
%  &=\lim_{n\rto}\frac{1}{(\log n)^k}\sum_{m\le j_1<...<j_k\le n}\frac{1}{j_1(j_2-j_1)\cdots(j_k-j_{k-1})}\no\\
%  &\quad\quad-\lim_{n\rto}\frac{1}{m(\log n)^{k}}\sum_{1\le j_1<\cdots <j_{k-1}\le n-m}\frac{1}{j_1(j_2-j_1)\cdots (j_{k-1}-j_{k-2})}\no\\
%  &=1.\no
%\end{align}
%The lemma is proved. \qed
\section{A version of the Borel-Cantelli lemma}\label{sec3}
The following Theorem \ref{fis} can be seen as a version of the Borel-Cantelli lemma. Its proof is indeed a summary of  Cs\'aki et al. \cite[Theorem 1.1]{cfrc} and James et al.
\cite[Theorem 3.1]{jlp}. But some details are different. So we provide a complete proof here. We remark that to prove Theorem \ref{main}, we just need to check one-by-one that all conditions of Theorem \ref{fis} are fulfilled.

\begin{theorem}[\cite{cfrc}, \cite{jlp}]\label{fis}
  For $n>m\ge 0,$ let $D(m)\ge 1$ and $D(m,n)\ge1$ be certain numbers.
Assume that $(\Omega,\mathcal F, \mathbb P)$ is a probability space and $\Gamma_n, n\ge 0$ are $\mathcal F$-measurable events such that \begin{align}\label{dng}
  \lim_{n\rto}D(n)\mathbb P(\Gamma_n)=\sigma
\end{align} for some $0<\sigma<\infty.$

(i) Suppose that $\sum_{n=2}^\infty\frac{1}{D(n)\log n}<\infty,$ \begin{align}\label{cki}
  \mathbb P(\Gamma_{k+1}^c\cdots \Gamma_n^c|\Gamma_i\Gamma_k)=\mathbb P(\Gamma_{k+1}^c\cdots \Gamma_n^c|\Gamma_k),\  0\le i\le k<n,
  \end{align}
and there exists a number $N_1>0$ such that
\begin{align}
  D(m,n)\leq c(n-m),\ \mathbb P(\Gamma_m|\Gamma_n)>\frac{c }{D(m,n)},\ n > m> N_1.\label{dxyc}
\end{align}
Then we have $\mathbb P\z(\bigcap_{i=0}^\infty \bigcup_{j=i}^\infty \Gamma_j\y)=0.$

(ii) Suppose that $\sum_{n=2}^\infty\frac{1}{D(n)\log n}<\infty,$ there exist $n_0>0$ and $\delta>0$ such that  $D(n)\le \delta n\log n,$ for all $n\ge n_0$ and \begin{align}\label{dmna}
  \frac{D(m,n)}{D(m)}=1-\prod_{i=m}^{n-1}\z(1-\frac{1}{D(i)}\y),\ n>m\ge 0.
\end{align}
In addition, assume that for each $\ve>0,$ there exists   $N>0$ such that $D(k),k\ge N$ is increasing in $k,$ and
\begin{align}
  \frac{D(m,n)}{D(m)}\frac{\mathbb P(\Gamma_m\Gamma_{n})}{\mathbb P(\Gamma_m)\mathbb P(\Gamma_{n})}
  \le1+\varepsilon, \  m>N,\ n-m>N.\label{ppt}
  \end{align}
  Then we have $\mathbb P\z(\bigcap_{i=0}^\infty \bigcup_{j=i}^\infty \Gamma_j\y)=1.$
\end{theorem}

\proof   We prove part (i) first. Suppose all conditions of part (i) are fulfilled.
 For $0\le m<k,$ set $A_{m,k}:=\sum_{j=2^{m}+1}^{2^k}1_{\Gamma_j},$ $l_m=\max\z\{2^{m}+1<j\le 2^{m+1}:\ 1_{\Gamma_j}=1\y\}$
 and write
$$  b_m:=\min_{k\in (2^{m},2^{m+1}]}\sum_{i=1}^{2^{m-1}} \frac{1}{D(k-i,k)}.
$$
Let $N_1$ be as in \eqref{dxyc}.
 Taking \eqref{cki} and \eqref{dxyc} into consideration, for  $m\ge 1+\frac{\log N_1}{\log 2}=:M_1$ (or equivalently $2^{m-1}>N_1$), $2^m<k\le  2^{m+1}$ and $2^{m-1}< i<k$ we have
\begin{align}
  \mathbb P&(\Gamma_i|A_{m,m+1}>0,l_m=k)=
  \frac{\mathbb P\z(\Gamma_i\Gamma_k\Gamma_{k+1}^c\cdots\Gamma_{2^{m+1}}^c\y)}{\mathbb P\z(\Gamma_k\Gamma_{k+1}^c\cdots\Gamma_{2^{m+1}}^c\y)}
    =\mathbb P(\Gamma_i|\Gamma_k)\ge \frac{c}{D(i,k)}.\no
  \end{align}
Thus, for  $m\ge M_1,$
       \begin{align}
        \sum_{j=2^{m-1}+1}^{2^{m+1}}&\mathbb P(\Gamma_j)=\mathbb E\z(A_{m-1,m+1}\y)\no\\
              &\ge \sum_{k=2^{m}+1}^{2^{m+1}}\mathbb P(A_{m,m+1}>0,l_m=k) \mathbb E(A_{m-1,m+1}|A_{m,m+1}>0,l_m=k)\no\\
              &= \sum_{k=2^{m}+1}^{2^{m+1}}\mathbb P(A_{m,m+1}>0,l_m=k) \sum_{i={2^{m-1}+1}}^kP(\Gamma_i|A_{m,m+1}>0,l_m=k)\no\\
              &\ge c \mathbb P(A_{m,m+1}>0)\min_{2^{m}< k \le 2^{m+1}} \sum_{i=1}^{2^{m-1}}\frac{1}{D(k-i,k)}\no\\
              &=c b_m \mathbb P(A_{m,m+1}>0).\label{pc}
      \end{align}
Using \eqref{dxyc}, we get $
  b_m\ge c\sum_{i=1}^{2^{m-1}}\frac{1}{i}\ge cm
$
for $m>M_1.$ Since $\mathbb P\z(\Gamma_n\y)\sim \sigma/D(n),$ as $n\rto,$
 thus,  from \eqref{pc}, we infer that
\begin{align*}
\sum_{m=M_1}^\infty &\mathbb P(A_{m,m+1}>0)\le c\sum_{m= M_1}^\infty\frac{1}{b_m}\sum_{j=2^{m-1}+1}^{2^{m+1}}\mathbb P(j\in C)\\
&\le c\sum_{m=M_1}^\infty\frac{1}{m}\sum_{j=2^{m-1}+1}^{2^{m+1}} \frac{1}{D(j)}\le c \sum_{m=1}^\infty \sum_{j=2^{m-1}+1}^{2^{m+1}} \frac{1}{D(j)\log j}\no\\
&\le c\sum_{n=2}^\infty \frac{1}{D(n)\log n}<\infty.
\end{align*}
Applying  the Borel-Cantelli lemma, we get  $\mathbb P\z(\bigcap_{n=1}^\infty\bigcup_{m=n}^\infty\{A_{m,m+1}>0\}\y)=0.$ Thus,  $\mathbb P\z(\bigcap_{i=0}^\infty \bigcup_{j=i}^\infty \Gamma_j\y)=0.$ Part (i) is proved.

%%%%%%%%%%%%%%%%%%%%%%%%%%%%%%%%%%%%%%%555

Next, we  prove  part (ii). Suppose all conditions of part (ii) hold.
For $k\ge 1,$ set
$m_k=[k\log k].$  Since $\mathbb P\z(\Gamma_n\y)\sim \sigma/D(n),$ as $n\rto,$ there exists a number $M_2>0$ such that $\mathbb P(\Gamma_{m_k})>\sigma/2$ for $k\ge M_2.$ Notice that $D(n)$ is increasing in $n>N.$ It follows from Cs\'aki et al. \cite[Lemma 2.2]{cfrc} that $\sum_{k=M_2}^{\infty}\frac{1}{D([k\log k])}$ and $\sum_{k=M_2}^{\infty}\frac{1}{D(k)\log k}$ are equiconvergent. Therefore, we get
\begin{align}
  \sum_{k=M_2}^{\infty}\mathbb P(\Gamma_{m_k})\ge \frac{\sigma}{2}\sum_{k=M_2}^{\infty}\frac{1}{D([k\log k])}=\infty. \label{syi}
\end{align}
Now fix $\varepsilon >0.$
Obviously,
 for $l>k,$
$m_l-m_k\ge l\log l-k\log k\ge \log k.$  Then, there exists a number $M_3>0$ such that for all $k\ge M_3,$
$m_k\ge N, m_l-m_k>N.$
Therefore, it follows from \eqref{dmna} and \eqref{ppt} that for $l>k>M_3,$
\begin{align}\label{e3}
\mathbb P(\Gamma_{m_k}\Gamma_{m_l})
&\leq (1+\varepsilon)\mathbb P(\Gamma_{m_k})\mathbb P(\Gamma_{m_l})\z(\frac{D(m_k,m_l)}{D(m_k)}\y)^{-1}\nonumber\\
&=(1+\varepsilon)\z(1-\prod_{i=m_k}^{m_l-1}\z(1-\frac{1}{D(i)}\y)\y)^{-1}\mathbb P(\Gamma_{m_k})\mathbb P(\Gamma_{m_l})\no\\
&\le(1+\varepsilon)\z(1-\exp\z\{-\sum_{i=m_k}^{m_l-1}\frac{1}{D(i)}\y\}\y)^{-1}\mathbb P(\Gamma_{m_k})\mathbb P(\Gamma_{m_l}).
\end{align}
Define
\begin{align}
\ell_1=\min\z\{l\geq k:\sum_{i=m_k}^{m_l-1}\frac{1}{D(i)}\geq \log \frac{1+\varepsilon}{\varepsilon}\y\}.\no
\end{align}
Then for $l\geq \ell_1,$ we have
$\z(1-\exp\z\{-\sum_{i=m_k}^{m_l-1}\frac{1}{D(i)}\y\}\y)^{-1}\leq 1+\varepsilon.$
Therefore, it follows from (\ref{e3}) that
\begin{align}
\mathbb P(\Gamma_{m_k}\Gamma_{m_l})\leq (1+\varepsilon)^2\mathbb P(\Gamma_{m_k})\mathbb P(\Gamma_{m_l}), \text{  for } l\geq \ell_1.\label{e6}
\end{align}
Next, we consider $k<l<\ell_1.$ Note that for $0\leq u\leq \log \frac{1+\varepsilon}{\varepsilon}$, we have $1-e^{-u}\geq cu$. Since $D(i),i\ge m_k$ is increasing in $i,$ thus, using again the fact $\mathbb P\z(\Gamma_n\y)\sim \sigma/D(n),$ as $n\rto,$ from (\ref{e3}) we deduce that
\begin{align}
\mathbb P(\Gamma_{m_k}\Gamma_{m_l})&\le(1+\varepsilon)\z(1-\exp\z\{-\sum_{i=m_k}^{m_l-1}\frac{1}{D(i)}\y\}\y)^{-1}\mathbb P(\Gamma_{m_k})\mathbb P(\Gamma_{m_l})\nonumber\\
&\leq\frac{cP(\Gamma_{m_k})\mathbb P(\Gamma_{m_l})}{\sum_{i=m_k}^{m_l-1}\frac{1}{D(i)}}
\leq \frac{cP(\Gamma_{m_k})\mathbb P(\Gamma_{m_l})D(m_l)}{m_l-m_k}\le \frac{cP(\Gamma_{m_k})}{l\log l-k\log k}.\no
\end{align}
Since $D(n)\le \delta n\log n,$ for all $n\ge n_0,$  it can be shown that
$
\frac{\log \ell_1}{\log k}\leq \gamma$
for some constant $\gamma$ depending only on $\varepsilon,$ see \cite[p.635, display (4.2)]{cfrc}.
Thus,
\begin{align}
\sum_{l=k+1}^{\ell_1-1}&\mathbb P(\Gamma_{m_k}\Gamma_{m_l})\leq cP(\Gamma_{m_k})\sum_{l=k+1}^{\ell_1-1}\frac{1}{l\log l-k\log k}\nonumber\\
&\leq cP(\Gamma_{m_k})\frac{1}{\log k}\sum_{l=k+1}^{\ell_1-1}\frac{1}{l-k}\leq cP(\Gamma_{m_k})\frac{\log \ell_1}{\log k}\le cP(\Gamma_{m_k}).\label{e8}
\end{align}
Taking (\ref{e6}) and (\ref{e8}) together, we have
$$\sum_{k=M_3}^{n}\sum_{l=k+1}^{n}\mathbb P(\Gamma_{m_k}\Gamma_{m_l})\leq\sum_{k=M_3}^{n}\sum_{l=k+1}^{n}(1+\varepsilon)^2\mathbb P(\Gamma_{m_k})\mathbb P(\Gamma_{m_l})+c\sum_{k=M_3}^{n}\mathbb P(\Gamma_{m_k}).$$
Writing $H(\varepsilon)=(1+\varepsilon)^2$, owing to (\ref{syi}) , we have
\begin{align*}
\alpha_{H(\varepsilon)}&:=\varliminf_{n\rightarrow\infty}\frac{\sum_{k=M_3}^{n}\sum_{l=k+1}^{n}\mathbb P(\Gamma_{m_k}\Gamma_{m_l})-\sum_{k=M_3}^{n}\sum_{l=k+1}^{n}H(\varepsilon)\mathbb P(\Gamma_{m_k})\mathbb P(\Gamma_{m_l})}{[\sum_{k=M_3}^{n}\mathbb P(\Gamma_{m_k})]^2}\nonumber\\
&\leq\lim_{n\rightarrow\infty}\frac{c}{\sum_{k=M_3}^{n}\mathbb P(\Gamma_{m_k})}=0.
\end{align*}
An application of a generalized version of the Borel-Cantelli lemma (see Petrov \cite[p. 235]{pe04}) yields that
$\mathbb P\z(\bigcap_{k=M_3}^\infty\bigcup_{j=k}^\infty\Gamma_{m_j}\y)\geq\frac{1}{H(\varepsilon)+2\alpha_{H(\varepsilon)}}\ge\frac{1}{(1+\varepsilon)^2}.$
As a consequence, we get
$
  \mathbb P\z(\bigcap_{i=0}^\infty \bigcup_{j=i}^\infty \Gamma_j\y)\ge \mathbb P\z(\bigcap_{k=M_3}^\infty\bigcup_{j=k}^\infty\Gamma_{m_j}\y) \ge\frac{1}{(1+\varepsilon)^2}.
$
Since $\varepsilon$ is arbitrary, letting $\varepsilon\rightarrow0,$ we conclude that $\mathbb P\z(\bigcap_{i=0}^\infty \bigcup_{j=i}^\infty \Gamma_j\y)=1.$ Part (ii)  is proved. \qed

\section{Upcrossing times represented by loops, moves and escapes}\label{sec4}

In this section, by decomposing the path of the walk, we represent the upcrossing times $\xi(n,\uparrow),n\ge 0$ by certain loops, moves and escapes. Throughout this section, we assume $D(0)<\infty.$ As a consequence, we see from Proposition \ref{crt} that $\lim_{n\rto}X_n=\infty$ almost surely.
\subsection{Loops, moves and escapes}

Now, considering the random walk $X,$ we introduce the following  events. For integers $0\le x<y,$
define \begin{align*}
  FL_x &=\z\{\begin{array}{l}
    \text{starting from }x, \text{ the walk jumps upward to }x+1, \text{ and then, }\\
     \text{after a number of steps, it returns to } x \text{ before going to }\infty
  \end{array}\y\},\\
FL_{xy}&=\z\{\begin{array}{l}
    \text{starting from }x, \text{ the walk jumps upward to }x+1, \text{ and then, }\\
     \text{after a number of steps, it returns to } x \text{ before it hits }y
  \end{array}\y\},\\
  F_{xy}&=\z\{\begin{array}{l}
    \text{starting from }x, \text{ the walk jumps upward to }x+1, \text{ and then, }\\
     \text{after a number of steps, it hits } y \text{ before  returning to }x
  \end{array}\y\},\\
  F_x&=\{\text{starting from }x, \text{ the walk jumps to } x+1 \text{ and never returns to }x\},
  \end{align*}
 and
  \begin{align*}
  BL_x&=\z\{\begin{array}{l}
    \text{starting from }x, \text{ the walk jumps downward to }x-1, \\
     \text{and then,}\text{ after a number of steps,}\text{ it returns to }x
  \end{array}\y\},\\
  BL_{yx} &=\z\{\begin{array}{l}
    \text{starting from }y, \text{ the walk jumps downward to }y-1, \text{ and }\\
     \text{then, after a number of steps, it returns to } y \text{ before hitting }x
  \end{array}\y\},\\
        B_{yx}&=\z\{\begin{array}{l}
    \text{starting from }y, \text{ the walk jumps downward to }y-1, \text{ and }\\
     \text{then, after a number of steps, it hits } x \text{ before  returning to }y
  \end{array}\y\}.
\end{align*}

\begin{figure}[ht]
  \centering
  \includegraphics[width=12cm]{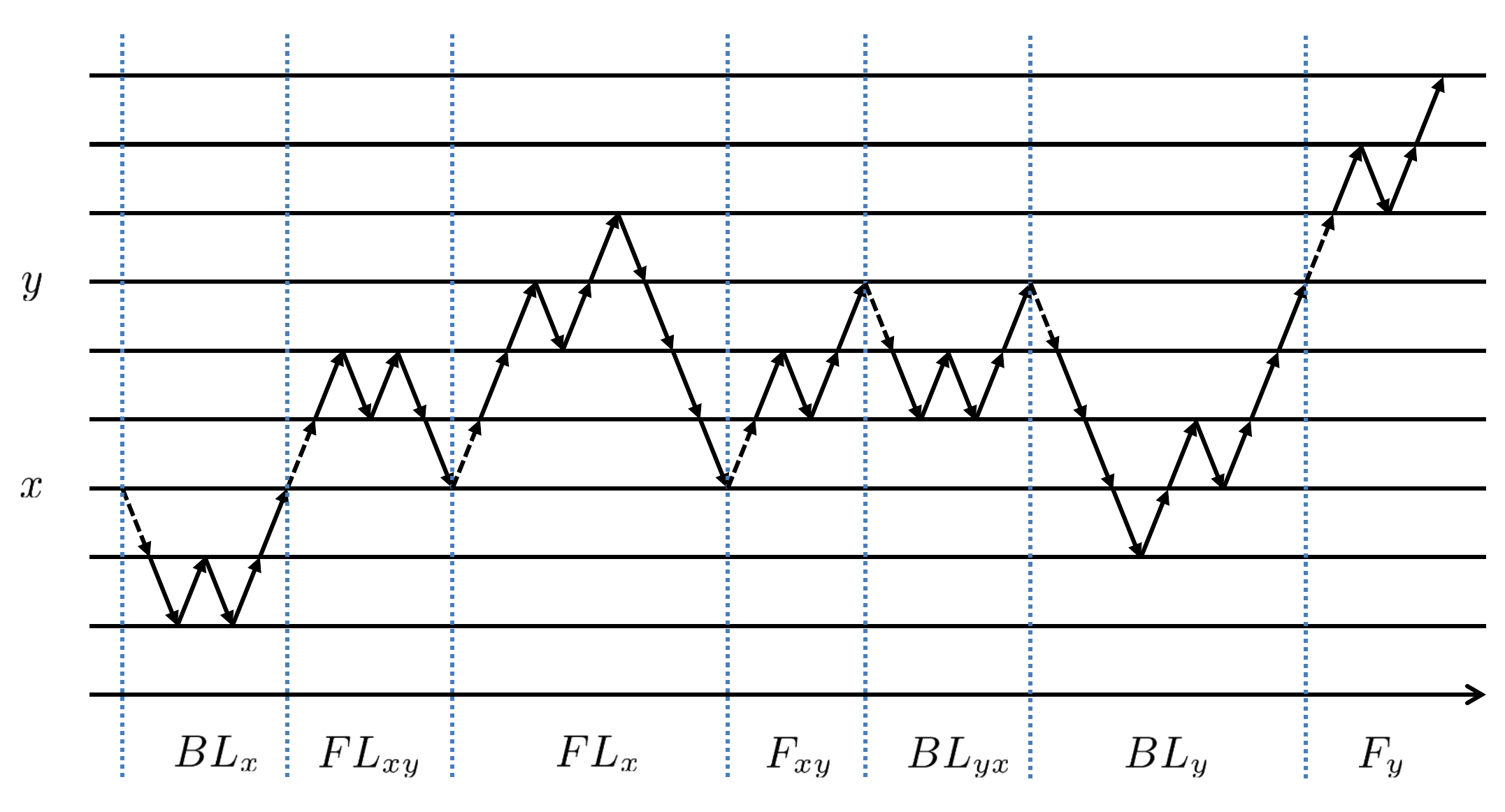}\\
  \vspace{-.2cm}

  \caption{ The figure illustrates the sample paths  of the loops, moves, and escapes defined in Definition \ref{dlm}. They all begin with a dashed arrow. Some special remarks should be added to $FL_x$ and $BL_y.$ Notice that the sample path of $FL_x$ in the figure has reached some site above $y.$ But by definition, for a forward loop $FL_x,$ the walk may or may not visit the sites above $y.$ The path of $BL_y$ should be understood similarly.}  \label{fig1}
\end{figure}
\begin{definition}\label{dlm}
  If the event $FL_x,$ $FL_{xy},$  $F_{xy},$ or $F_x$ occurs, we say that a forward loop  at $x,$  a forward loop at $x$ avoiding $y,$  a forward move from $x$ to $y,$ or an escape from $x$ to $\infty$ appears, respectively.
  Similarly, if the event $BL_{x},$ $BL_{yx},$ or $B_{yx}$ occurs, we say that a backward loop at $x,$ a backward loop at $y$ avoiding $x,$ or  a backward move from $y$ to $x$ appears, respectively. When there is no danger of confusion, these notations of events are also used to represent the loop, move, or escape that they define. For example, if we say there is a forward loop $FL_x,$ we mean that a forward loop appears at $x.$
\end{definition}

The sample paths of these loops, moves and escapes are illustrated in Figure \ref{fig1}.
 Consulting \eqref{pmnm}, \eqref{pmnp} and \eqref{ep} time and again, we get
\begin{align}
\mathbb P(FL_x)&=p_xP_{x+1}(x,\infty,-)=p_x\z(1-\frac{1}{D(x)}\y),\label{pflx}\\
    \mathbb P(FL_{xy})&=p_xP_{x+1}(x,y,-)=p_x\z(1-\frac{1}{D(x,y)}\y),\label{pflxy}\\
  \mathbb P(F_{xy})&=p_xP_{x+1}(x,y,+)=\frac{p_x}{D(x,y)},\label{pfxy}\\
  \mathbb P(F_x)&=p_xP_{x+1}(x,\infty,+)=\frac{p_x}{D(x)},\label{pfy}
\end{align}
and
\begin{align}
\mathbb P(BL_x)&=q_x,\label{pblx}\\
  \mathbb P(BL_{yx})&=q_yP_{y-1}(x,y,+)=\frac{q_y D(x,y-1)}{D(x,y)},\label{pblyx}\\
    \mathbb P(B_{yx})&=q_yP_{y-1}(x,y,-)=q_y\z(1-\frac{D(x,y-1)}{D(x,y)}\y).\label{pbyx}
  \end{align}

\subsection{Upcrossing times and branching processes with immigration}

Since the chain $X$ is transient, for each $x\in \mathbb Z_+,$ there must be an escape $F_x$ at $x.$ For $x\ge0,$ $k\ge1,$ if $\xi(x,\uparrow)=k,$ then there must be $k-1$ forward loops $FL_x$ and one escape $F_x$ at $x.$

Fixing $x\ge 1$ and $k\ge0,$ on the event $\{\xi(x-1,\uparrow)=k+1\},$ let
\begin{align}\label{det}
  \eta(x,i)=&\text{ the number of forward loops }FL_{x}\text{ at }x\text{ contained}\no\\
   &\text{ in the }i\text{-th forward loop }FL_{x-1}\text{ at } x-1, \text{ for }i\ge1
\end{align}
and
\begin{align}\label{dzt}
  \zeta(x)=&\text{ the number of forward loops }FL_{x}\text{ at }x\text{ contained}\no\\
   &\text{ in the escape }F_{x-1}\text{ at } x-1.
\end{align}
\begin{figure}
  \centering
  \includegraphics[width=12cm]{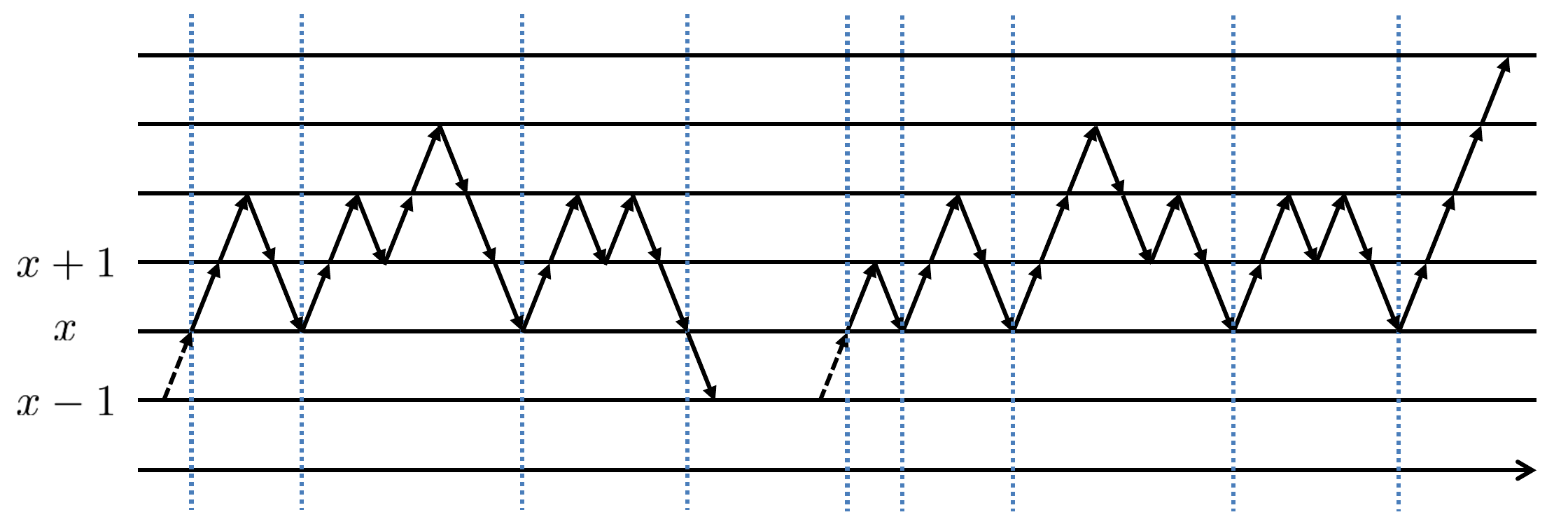}\\
  \vspace{-.1cm}

  \caption{The left part of the figure is the sample path of a forward loop $FL_{x-1}$ at $x-1,$ it contains $3$ forward loops $FL_{x}$ at $x.$ The right part is the sample path of an escape $F_{x-1}$ from $x-1$ to $\infty.$ It contains $4$ forward loops $FL_{x}$ at $x.$    }\label{fig2}
\end{figure}
Conditioned on the event $\{\xi(x-1,\uparrow)=k+1\},$ by strong Markov property, the paths contained in those $k$ forward loops $FL_{x-1}$ are  i.i.d. and moreover, they are all independent of the path contained in the  escape $F_{x-1}.$ Therefore, we conclude that
\begin{align}\label{xikr}
  \xi(x,\uparrow)=\sum_{i=1}^{\xi(x-1,\uparrow)-1}\eta(x,i)+\zeta(x)+1, \end{align}
where $\eta(x,i), i\ge1$ are i.i.d. and they are all independent of $\zeta(x),$ the number $1$  represents the escape $F_x$ from $x$ to $\infty.$
Indeed, from the definition in \eqref{det} and \eqref{dzt}, by the strong Markov property, $\eta(x,i),\zeta(x),x\ge1,i\ge1$ are mutually independent. Thus, \eqref{xikr} says indeed that
\begin{align}\label{xm}
  \{\xi(x,\uparrow)\}_{x=0}^\infty \text{ is a Markov chain (or a branching process with immigration).}
\end{align}

To study the weak cutpoints, we need to introduce a new chain related to the upcrossing times.  Let $Y_0=0,$ and for $x\ge 1,$ let
\begin{align*}
  Y_x=&\text{ number of forward loops }FL_{x}\text{ at }x\text{ contained}\\
  &\text{ in all forward loops }FL_{x-1}\text{ at }x-1.
\end{align*}
Then, in view of \eqref{det}, \eqref{dzt} and \eqref{xikr}, we have
\begin{align}\label{yyz}
  Y_x=\sum_{i=1}^{Y_{x-1}+\zeta(x-1)}\eta(x,i),\ x\ge1.
\end{align}
Clearly, by the strong Markov property, $Y_{x-1}$ and $\zeta(x-1)$ are mutually independent.  Thus, we come to the conclusion that
\begin{align}\label{ym}
  \{Y_x\}_{x=0}^\infty \text{ is a Markov chain (or a branching process with immigration).}
\end{align}

Now, we consider the sets $C(a,b),$ $C_w$ and $C(*,a)$  defined in Definition \ref{def2} and \eqref{cstar}.
Clearly,
\begin{align}
  &\{x\in C(*,a)\}=\{\xi(x,\uparrow)=a\},\label{csa}\\
  &\{x\in C(a,b)\}=\{\xi(x,\uparrow)=b,\  \xi(x-1,\uparrow)=a-b+1\},\label{cabx}\\
  &\{x\in C_w\}=\{Y_x=0\}.\label{cwy}
\end{align}
From \eqref{xm} and \eqref{ym} we see that both $\{\xi(x,\uparrow)\}_{x=0}^\infty$ and $\{Y_x\}_{x=0}^\infty$ are Markov chains.
Thus, by the Markov property, we get directly the following lemma.
\begin{lemma}\label{lemm}
  Suppose $D(1)<\infty.$ Then we have
  \begin{align*}
  \mathbb P( j\notin & C, k+1\le j\le n|i\in C,k\in C)=\mathbb P( j\notin  C, k+1\le j\le n|k\in C),
    %\no\\
%    \mathbb P( j\notin & C(*,a), k+1\le j\le n|i\in C(*,a),k\in C(*,a))\no\\
%    &=\mathbb P( j\notin  C(*,a), k+1\le j\le n|k\in C(*,a)),\ 0\le i\le k<n,\no\\
%    \mathbb P( j\notin & C_w, k+1\le j\le n|i\in C_w,k\in C_w)\no\\
%    &=\mathbb P( j\notin  C_w, k+1\le j\le n|k\in C_w),\ 0\le i\le k<n.\no
  \end{align*} for $0\le i\le k<n,$ $C=C(a,b),$ $C(*,a)$ or $C_w.$
\end{lemma}
What we really need is Lemma \ref{lemm}. Of course, we can say more about the Markov chain $\{\xi(x,\uparrow)\}_{x=0}^\infty$ and $\{Y_x\}_{x=0}^\infty.$ For $x\ge0,$ write
\begin{align}
 f_x(s):=\frac{1}{P_x(x-1,\infty,-)}\frac{q_x}{1-sp_xP_{x+1}(x,\infty,-)},\ s\in[0,1],\label{fns}\\
 g_x(s):=\frac{1}{P_x(x-1,\infty,+)}\frac{p_xP_{x+1}(x,\infty,+)}{1-sp_xP_{x+1}(x,\infty,-)},\ s\in[0,1].\label{gns} \end{align}
By the Markov property, it follows that
 \begin{align}
   P_x(x-1,\infty,-)&=p_xP_{x+1}(x,\infty,-)P_{x}(x-1,\infty,-)+q_x,\no\\
   P_{x}(x-1,\infty,+)&=p_xP_{x+1}(x,\infty,+),\no
 \end{align} from which we get $f_x(1)=g_x(1)=1.$
Thus, both $f_x(s)$ and $g_x(s)$ are candidates for probability generation functions. By checking carefully, we can show that $f_x(s)$ and $g_x(s)$ are indeed the probability generation functions of the random variables $\eta(x,i)$ and $\zeta(x)$ defined in \eqref{det} and \eqref{dzt}, respectively. Consequently, from \eqref{xikr} and \eqref{yyz} we get the following proposition.
\begin{proposition}\label{pryx} Let $\xi(-1,\uparrow)=1$ and $f_x(s),$ $g_x(s),$ $x\ge0$ be as in \eqref{fns} and \eqref{gns}.  Then $\{\xi(x,\uparrow)\}_{x=-1}^\infty$ and $\{Y_x\}_{x=0}^\infty$ are branching processes with immigration such that
  \begin{align}
  &\mathbb E\z(s^{\xi(x,\uparrow)}\Big|\xi(x-1,\uparrow)=k+1\y)=s g_x(s)\z(f_x(s)\y)^{k},\ k\ge0,\ x\ge0,\no\\
    & \mathbb E\z(s^{Y_x}\big|Y_{x-1}=k\y)=g_{x-1}(f_x(s))(f_x(s))^k,\ k\ge 0,\ x\ge1.\no
\end{align}
\end{proposition}

 Proposition \ref{pryx} can be seen as a generalization of the classical Ray-Knight theorem (see Knight \cite[Theorem 1.1]{kn63}) on the local time of a recurrent simple random walk or the branching structure in the path of a transient simple random walk (see Dwass \cite[Theorem 2]{dw75}).

\section{Probability, joint probability and dependence}\label{sec5}
As seen in \eqref{dxyc} and \eqref{ppt}, to apply Theorem \ref{fis}, we need to compute the probability of $\{x\in C\}$ and  characterize the dependence of  $\{x\in C\}$ and $\{y\in C\}$ for $C=C(a,b),$ $C_w$ or $C(*,a).$ The following proposition is the main result of this section.
\begin{proposition}\label{propc}
  Suppose  $\rho_k\rightarrow 1$ as $k\rto.$ Let $C=C(a,b),C(*,a)$ or $C_w.$ Then
  \begin{align}
    \lim_{x\rto}\mathbb P(x\in C)D(x)=\lambda_C=
    \z\{\begin{array}{cl}
      2,& \text{if } C=C_w,\\
      \binom{a-1}{b-1}\frac{1}{2^a}, &\text{if }C(a,b), \\
      1, &\text{if }C(*,a),
    \end{array}\y.\label{lpc}
  \end{align}
  and
  \begin{gather}\label{wca}
 \mathbb P(y\in C|x\in C)\le \frac{c D(x)}{D(x,y)D(y)},\ y\ge x\ge0,\\
 \label{pdl} \mathbb P(x\in C|y\in C)\ge \frac{c}{D(x,y)},\ y\ge x\ge0,\\
  \mathbb P(j_1\in C,...,j_k\in C)\le \frac{c}{\z(\prod_{i=1}^{k-1}D(j_i,j_{i+1})\y)D(j_k)},\  0<j_1<...<j_k.\label{jkbw}
\end{gather} Furthermore, there exists $N>0$ such that for all $ x>N,y>x+N,$ we have
\begin{gather} 1-\varepsilon\le \frac{D(x,y)}{D(x)}\frac{\mathbb P(x\in C,y\in C)}{\mathbb P(x\in C)\mathbb P(y\in C)}
  \le1+\varepsilon,\label{pw}\\
 (\lambda_C-\ve)\frac{D(x)}{D(x,y)D(y)}\le \mathbb P(y\in C|x\in C)\le (\lambda_C+\ve)\frac{D(x)}{D(x,y)D(y)},  \label{wc}
\end{gather}
and for $0=j_0<j_1<...<j_k,$ $j_i-j_{i-1}>N,$ $i=1,...,k$ we have
\begin{align}\label{jkaw}
\frac{(\lambda_C-\varepsilon)^k}{D(j_k)\prod_{i=1}^{k-1}D(j_i,j_{i+1})} &\le \mathbb P(j_1\in C,...,j_k\in C)\le \frac{(\lambda_C+\varepsilon)^k}{D(j_k)\prod_{i=1}^{k-1}D(j_i,j_{i+1})}.
\end{align}
\end{proposition}

\subsection{Weak cutpoints-Proof of Proposition \ref{propc} for $C=C_w$}
To warm up, we consider first the weak cutpoints. The following lemma gives the probability of $\{x\in C_w\}$ and the joint probability of $\{x\in C_w, y\in C_w\}.$
\begin{lemma}\label{pjpw}
  We have
\begin{align}
  &\mathbb P(x\in C_w)=\frac{1}{p_xD(x-1)}\label{xcw},\\
\label{xycw}
  &\mathbb P(x\in C_w,y\in C_w)=\frac{1}{p_xD(x-1,y)}\frac{1}{1-\frac{q_yD(x-1,y-1)}{D(x-1,y)}}\frac{1}{D(y-1)}.
\end{align}
\end{lemma}

\proof
To begin with, we prove \eqref{xcw}.
Notice that the event $\{x\in C_w\}$ if and only if after hitting $x$ for the first time,  the walk forms a number of (possibly $0$) backward loops $BL_x$ at $x$ one after another and then never visits $x-1.$ Thus, on accounting of \eqref{pblx}, we get
\begin{align*}
  \mathbb P(x\in C_w)=\sum_{k=0}^\infty (\mathbb P(BL_x))^k P_x(x-1,\infty,+)=\sum_{k=0}^\infty\frac{q_x^k }{D(x-1)}=\frac{1}{p_xD(x-1)},
\end{align*}
which proves \eqref{xcw}.

Next, we show \eqref{xycw}. Notice  that the event $\{x\in C_w,y\in C_w\}$ occurs if and only if the following events occur consecutively: after hitting $x$ for the first time, the walk forms a number of (possibly $0$) backward loops $BL_x$ at $x$ one after another; then, restarting from $x,$ it hits $y$ before it hits $x-1;$ restarting from $y,$ it forms  a number of (possibly $0$) backward loops $BL_{y(x-1)}$ at $y$ avoiding $x-1$  one after another; finally, restarting from $y,$ it never hits  $y-1.$

Therefore, it follows from the Markov property and \eqref{pmnm}-\eqref{ep} and \eqref{pblx}-\eqref{pblyx} that
\begin{align}
  \mathbb P&(x\in C_w,y\in C_w)\no\\
  &=\sum_{k=0}^\infty \sum_{j=0}^\infty (\mathbb P(BL_x))^kP_x(x-1,y,+)\z(P(BL_{y(x-1)})\y)^jP_{y}(y-1,\infty,+)\no\\
  &=\sum_{k=0}^\infty \sum_{j=0}^\infty q_x^k\frac{1}{D(x-1,y)}\z(\frac{q_yD(x-1,y-1)}{D(x-1,y)}\y)^j\frac{1}{D(y-1)}\no\\
  &=\frac{1}{p_xD(x-1,y)}\frac{1}{1-\frac{q_yD(x-1,y-1)}{D(x-1,y)}}\frac{1}{D(y-1)},\no
\end{align}
which finishes the proof of \eqref{xycw}. \qed

Next, we give the proof of Proposition \ref{propc} for $C=C_w.$
\vspace{.5cm}

\noindent{\it Proof of Proposition \ref{propc} for $C=C_w.$} Let $C=C_w.$ Suppose $\rho_k\rightarrow 1$ as $k\rto.$ Then  we have $p_x\rightarrow 1/2,$  $D(x)\rto $ as $x\rto$  and $p_k\vee p_k^{-1}<c,$ $\rho_k\vee \rho_k^{-1}\le c,$ for all $k\ge1.$  Thus, it follows from \eqref{xcw} that
\begin{align}
  \mathbb P(x\in C_w)D(x)=\frac{D(x)}{p_xD(x-1)}=\frac{1}{p_x\z(\rho_x+\frac{1}{D(x)}\y)}\rightarrow 2, \text{ as }x\rto,\label{pdw}
\end{align}
which proves \eqref{lpc}.

 Next, taking \eqref{xcw} and \eqref{xycw} together, for $y\ge x\ge 1,$ we get
\begin{align}
 \mathbb P(y&\in C_w|x\in C_w)=\frac{D(x-1)}{D(x-1,y)D(y-1)}\frac{1}{1-\frac{q_yD(x-1,y-1)}{D(x-1,y)}}\no\\
 &=\frac{D(x)}{D(x,y)D(y)}\frac{D(x-1)}{D(x)}\frac{D(x,y)}{D(x-1,y)}\frac{D(y)}{D(y-1)}\frac{1}{1-\frac{q_yD(x-1,y-1)}{D(x-1,y)}}\no\\
&=\frac{D(x)}{D(x,y)D(y)}\z(\rho_x+\frac{1}{D(x)}\y)\frac{1}{\rho_x+\frac{1}{D(x,y)}}\frac{1}{\rho_y+\frac{1}{D(y)}}\frac{1}{1-\frac{q_yD(x-1,y-1)}{D(x-1,y)}},\label{pcxy}
\end{align}
and
\begin{align}
  \mathbb P(x&\in C_w|y\in C_w)=\frac{1}{p_xD(x-1,y)}\frac{p_y}{1-q_y\frac{D(x-1,y-1)}{D(x-1,y)}}\no\\
  &\ge \frac{p_y}{p_xD(x-1,y)}=\frac{p_y}{p_xD(x,y)}\frac{1}{\rho_x+\frac{1}{D(x,y)}}\no\\
  &\ge \frac{p_y}{p_xD(x,y)}\frac{1}{1+\rho_x}=\frac{p_y}{D(x,y)}\ge \frac{c}{D(x,y)}.\no
\end{align}
We thus get \eqref{pdl}.
Since $D(m)\ge D(m,n)\ge 1$ for  $n> m\ge 0$ and $ D(x-1,y)>D(x-1,y-1)$  for $y\ge x\ge1,$
 then from \eqref{pcxy}, we obtain
\begin{align}
  \mathbb P(y\in C_w|x\in C_w)\le \frac{D(x)}{D(x,y)D(y)}\frac{\rho_x+1}{\rho_x}\frac{\rho_y+1}{\rho_y} \no \end{align}
 from which we get \eqref{wca}.
 Notice that by \eqref{pdw} we have $\mathbb P(x\in C_w)\le \frac{c}{D(x)}.$
Thus, due to \eqref{ym} and \eqref{cwy}, \eqref{jkbw} is a direct consequence of \eqref{wca}.

Now  fix $\eta>0.$ By \eqref{pdw}, there exists a number $M_1>0$ such that
\begin{align}
   (2-\eta)\frac{1}{D(x)}&\le \mathbb P(x\in C_w)\le (2+\eta)\frac{1}{D(x)},\  x\ge M_1.\label{ulcw}
\end{align}
Using again \eqref{xcw} and \eqref{xycw}, we get
\begin{align}
  &\frac{\mathbb P(x\in C_w,y\in C_w)}{\mathbb P(x\in C_w)\mathbb P(y\in C_w)}=\frac{D(x-1)}{D(x-1,y)}\frac{p_{y}}{1-\frac{q_{y}D(x-1,y-1)}{D(x-1,y)}}\no\\
  &=\frac{D(x)}{D(x,y)}\frac{D(x-1)}{D(x)}\frac{D(x,y)}{D(x-1,y)}\frac{p_{y}}{1-\frac{q_{y}D(x-1,y-1)}{D(x-1,y)}}.\no
\end{align}
Applying Lemma \ref{edxy}, we can find a number $M_2>0$ such
\begin{align*}
 1-\eta\le  \frac{D(x-1)}{D(x)}\frac{D(x,y)}{D(x-1,y)}\frac{p_{y}}{1-\frac{q_{y}D(x-1,y-1)}{D(x-1,y)}}\le 1+\eta,
\end{align*}
for all $ x>M_2,\ y>x+M_2.$ Therefore, we get
\begin{align} 1-\eta\le \frac{D(x,y)}{D(x)}\frac{\mathbb P(x\in C_w,y\in C_w)}{\mathbb P(x\in C_w)\mathbb P(y\in C_w)}
  \le1+\eta,\ x>M_2,\ y>x+M_2.\label{uljw}
  \end{align}
  Taking \eqref{ulcw} and \eqref{uljw} together, we infer that
  \begin{align} \label{ulcp}
 \frac{(2-\eta)(1-\eta)D(x)}{D(x,y)D(y)}\le \mathbb P(y\in C|x\in C)\le \frac{(2+\eta)(1+\eta)D(x)}{D(x,y)D(y)},
\end{align}
 for $ x>M_1\vee M_2,\ y>x+M_1\vee M_2.$
Choosing $\eta$ small enough such that $(2+\eta)(1+\eta)< 2+\ve$ and $(2-\eta)(1-\eta)> 2-\ve$ and letting $N=M_1\vee M_2,$ from \eqref{uljw} and \eqref{ulcp}, we get \eqref{pw} and \eqref{wc}, respectively. Finally, putting \eqref{ulcw} and \eqref{ulcp} together, we get \eqref{jkaw}. Proposition \ref{propc} is proved for $C=C_w.$ \qed

\subsection{Local times and upcrossing times-Proof of Proposition \ref{propc} for $C=C(a,b)$ or $C(*,a)$}
Compared with the case $C=C_w,$ it is much more involved to prove Proposition \ref{propc} for $C=C(a,b)$ or $C(*,a).$ The main difficulty is to compute the joint probability of $\{x\in C, y\in C\}$ for $C=C(a,b)$ or $C(*,a).$

\subsubsection{Probabilities and joint probabilities}
To begin with, we compute the probabilities of $\{x\in C(a,b)\}$ and $\{x\in C(*,b)\}.$

\begin{lemma}\label{lxa}
    For $x\ge1$ and $a\ge b\ge 1,$ we have
  \begin{align}\label{dxa}
  &\mathbb P(\xi(x)=a,\xi(x,\uparrow)=b)=\binom{a-1}{b-1}\frac{p_x^bq_x^{a-b}}{D(x)}\z(1-\frac{1}{D(x)}\y)^{b-1},\\
  &\mathbb P(\xi(x,\uparrow)=b)=\frac{1}{D(x)}\z(1-\frac{1}{D(x)}\y)^{b-1}.\label{dxb}
  \end{align}
 \end{lemma}

\proof  For  $x\ge1$ and $a\ge b\ge 1,$ $\{\xi(x)=a,\xi(x,\uparrow)=b\}$ if and only if the following events occur  consecutively:
starting from $0,$ the walk hits $x$ for the first time (with probability 1); restarting from $x,$ the walk forms $a-b$ backward loops $BL_x$ at $x,$ and $b-1$ forward loops $FL_x$ at $x;$ finally, the walk forms an escape $F_x$ from $x$ to $\infty.$

 We next determine the possible orders of those $a-b$ backward loops $BL_x,$ $b-1$ forward loops $FL_x$ and one escape $F_x.$
To this end, notice that there are $a$ jumps of the walk from $x$ to $x+1$ or $x-1.$ We know exactly that when standing at $x$ for the last time, the walk must take a jump from $x$ to $x+1,$ which forms the beginning of the escape $F_x$ at $x.$  But there are $\binom{a-1}{b-1}$ possible ways to choose $b-1$ jumps from the other $a-1$ jumps to jump from $x$ to $x+1,$ which are the beginnings of those $b-1$ forward loops $FL_x.$

To sum up the above discussion,  owing to \eqref{pflx}, \eqref{pfy} and \eqref{pblx}, we get
 \begin{align}
  \mathbb P(\xi&(x)=a,\xi(x,\uparrow)=b)=\binom{a-1}{b-1}(\mathbb P(BL_x))^{a-b}(\mathbb P(FL_x))^{b-1}\mathbb P(F_x)\no\\
  &=\binom{a-1}{b-1}q_x^{a-b}\z(p_xP_{x+1}(x,\infty,-)\y)^{b-1}p_xP_{x+1}(x,\infty,+)\no\\
    &=\binom{a-1}{b-1}\frac{p_x^bq_x^{a-b}}{D(x)}\z(1-\frac{1}{D(x)}\y)^{b-1}.\no
  \end{align}
  Finally, taking summation on both sides of \eqref{dxa} over $a$ from $b$ to $\infty,$ we obtain \eqref{dxb}. The lemma is proved. \qed

The lemma below gives the joint probability of  $\{x\in C(a,b), y\in C(n,m)\},$ or $\{x\in C(*,b),y\in C(*,m)\}.$
%%%%%%%%%%%%%%%%%%%%%%%%%%%%%%%%%%%%%%%%%%%%%%%%%%%%%%%%%%%%%%%%

\begin{lemma}\label{aaxy}
  For integers  $1\le b\le a,$ $1\le m\le n$ and $1\le x<y,$    we have
\begin{align}\label{rolj}
  \mathbb P&(\xi(x)=a,\xi(x,\uparrow)=b,\xi(y)=n,\xi(y,\uparrow)=m)\no\\
     &=\binom{a-1}{b-1}\sum_{i=1}^{b\wedge (n-m+1)}\binom{b-1}{i-1}\binom{n-i}{m-1}\binom{n-1}{i-1}
   q_x^{a-b}\frac{p_y}{D(y)}\no\\
  &\times \z(p_x\z(1-\frac{1}{D(x,y)}\y)\y)^{b-i} \z(\frac{p_x}{D(x,y)}\y)^{i}\z(p_y\z(1-\frac{1}{D(y)}\y)\y)^{m-1}\no\\
  &\times \z(q_y\z(\frac{D(x,y-1)}{D(x,y)}\y)\y)^{n-m-(i-1)}
   \z(q_y\z(1-\frac{D(x,y-1)}{D(x,y)}\y)\y)^{i-1}
\end{align}
and
\begin{align}
\mathbb P&(\xi(x,\uparrow)=b,\xi(y,\uparrow)=m)\no\\
  &=\sum_{i=1}^{b}\binom{b-1}{i-1}\binom{m+i-2}{i-1}\z(\z(1-\frac{1}{D(x,y)}\y)\y)^{b-i} \z(\frac{1}{D(x,y)}\y)^{i}\no\\
  &\times \z(\frac{p_y\z(1-\frac{1}{D(y)}\y)}{1-q_y\frac{D(x,y-1)}{D(x,y)}}\y)^{m-1}
   \z(\frac{q_y\z(1-\frac{D(x,y-1)}{D(x,y)}\y)}{1-q_y\frac{D(x,y-1)}{D(x,y)}}\y)^{i-1}\frac{\frac{p_y}{D(y)}}{1-q_y\frac{D(x,y-1)}{D(x,y)}}.\label{jpu}
\end{align}
 \end{lemma}
\begin{remark}
  Setting $n=a$ and taking summation on  both sides of \eqref{rolj} over $b$ and $m$ from $1$ to $a,$ by some tedious manipulation, we get
  \begin{align*}
   \mathbb P(\xi&(x)=a,\xi(y)=a)\no\\
   &=\sum_{i=1}^a \binom{a-1}{i-1}^2 \z(\frac{p_x}{D(x,y)}\y)^{i}\z(1-\frac{p_x}{D(x,y)}\y)^{a-i}
   \z(q_y\z(1-\frac{D(x,y-1)}{D(x,y)}\y)\y)^{i-1}\no\\
   &\quad\quad\times \z(q_y\frac{D(x,y-1)}{D(x,y)}+p_y\z(1-\frac{1}{D(y)}\y)\y)^{a-i}\frac{p_y}{D(y)}.
 \end{align*}
  But we do not need such a formula because
 $C(a,*)$ is finite  if and only if $C(a,b)$ is finite for all $b\in \{1,...,a\},$ while $C(a,*)$  is infinite if and only if $C(a,b)$ is infinite for some $b\in \{1,...,a\}.$

 The formula in \eqref{jpu} looks heavy. If we take summation on  both sides of \eqref{jpu} over $m$ from $1$ to $\infty,$ then by some careful computation, we get $\mathbb P(\xi(x,\uparrow)=b)=\z(1-\frac{1}{D(x)}\y)^{b-1}\frac{1}{D(x)},b\ge1,$ which coincides with \eqref{dxb}.
\end{remark}
 We give two proofs of Lemma \ref{aaxy} from different scopes, i.e., the combinatorial one and the path decomposition one.
 \vspace{.3cm}

 \noindent{\it Proof of Lemma \ref{aaxy}(Combinatorial scope).}
  For simplicity, write
 \begin{align*}
   &\mathcal A=\{\xi(x)=a,\xi(x,\uparrow)=b,\xi(y)=n,\xi(y,\uparrow)=m\}.
 \end{align*}
 Let $FL_{xy},$ $F_{xy},$ $FL_{y},$ $F_y$ and $BL_x,$ $BL_{yx},$ $B_{yx}$  be the loops, moves and escapes defined in Definition \ref{dlm} whose probabilities are given in \eqref{pflx}-\eqref{pbyx}.

Notice that
the event $\mathcal A$ occurs if and only if for some $i,$
there are $a-b$ backward loops $BL_x$ at $x,$ $b-i$ forward loops $FL_{xy}$ at $x$ avoiding $y,$ $i$ forward moves $F_{xy}$ from $x$ to $y$, $m-1$ forward loops  $FL_{y}$ at $y,$  $n-m-(i-1)$ backward loops  $BL_{yx}$ at $y$ avoiding $x,$ $i-1$ backward moves $B_{yx}$ from $y$ to $x,$ and finally an escape $F_y$ from $y$ to $\infty.$ See Figure \ref{fig3} for the case $a=9,$ $n=8,$ $b=5,$ $m=4$ and $i=3.$
\begin{figure}[ht]
  \centering
  \includegraphics[width=13cm]{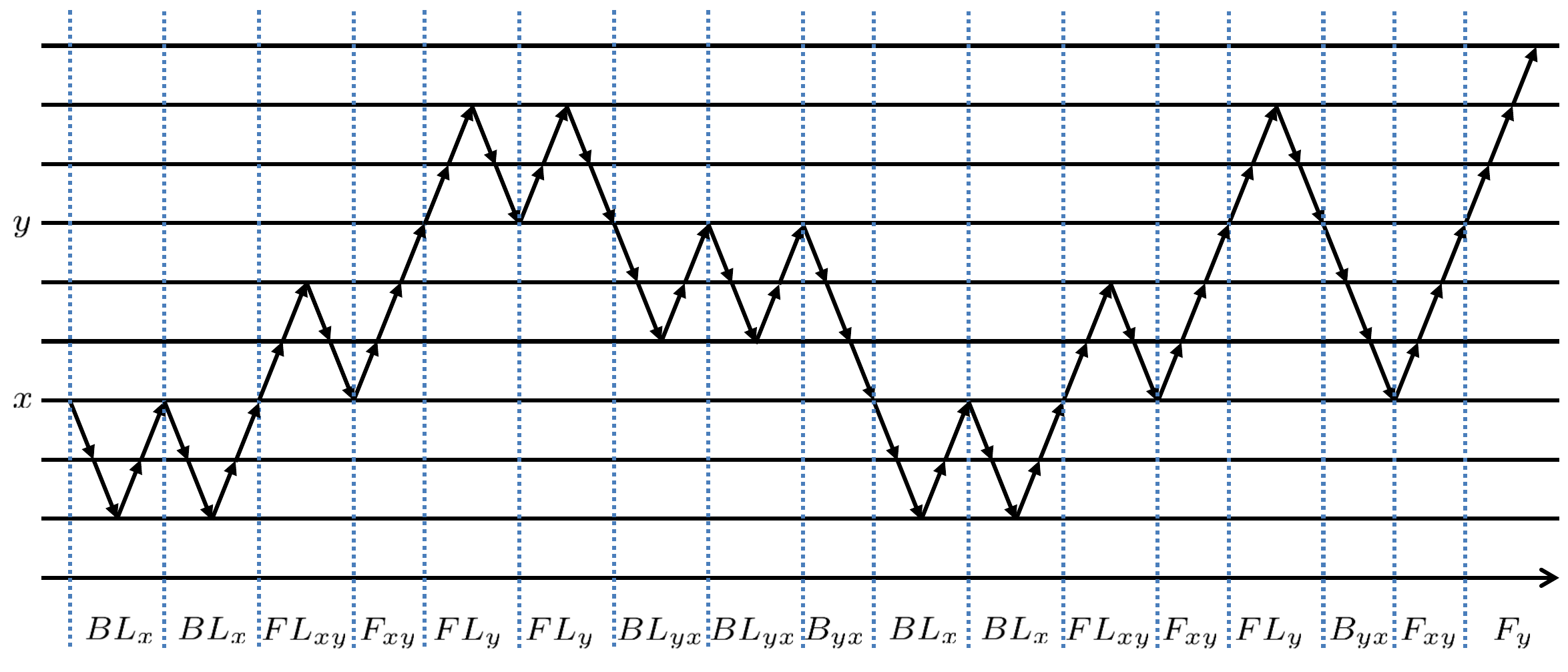}\\
  \caption{Suppose $a=9,$ $n=8,$ $b=5$ and $m=4.$ On the event $\mathcal A,$ if $i=3,$ then, there must be $4$ backward loops $BL_x$ at $x,$ $2$ forward loops $FL_{xy}$ at $x$ avoiding $y,$ $3$ forward moves $F_{xy}$ from $x$ to $y$, $3$ forward loops  $FL_{y}$ at $y,$  $2$ backward loops  $BL_{yx}$ at $y$ avoiding $x,$ $2$ backward moves $B_{yx}$ from $y$ to $x,$ and finally an escape $F_y$ from $y$ to $\infty.$ This figure is also helpful to understand the proof of \eqref{no}.}\label{fig3}
\end{figure}

Clearly, \begin{align}
  1\le i\le (n-m+1)\wedge b.\label{oi}
\end{align} If the order of those events is determined, in view of \eqref{pflx}-\eqref{pbyx}, by the Marov
property,  the probability of these events occurring one by one is
\begin{align}\label{pml}
  &q_x^{a-b}\z(p_x\z(1-\frac{1}{D(x,y)}\y)\y)^{b-i} \z(\frac{p_x}{D(x,y)}\y)^{i}\z(p_y\z(1-\frac{1}{D(y)}\y)\y)^{m-1}\no\\
  &\times \z(q_y\z(\frac{D(x,y-1)}{D(x,y)}\y)\y)^{n-(m-1)-i}
   \z(q_y\z(1-\frac{D(x,y-1)}{D(x,y)}\y)\y)^{i-1}\frac{p_y}{D(y)}
\end{align}
which does not depend on the order of these events.

Next, fix such a number $i.$ We claim that \begin{equation}\label{no}\begin{split}
  &\text{the possible number of the orders of these moves, loops,}\\
   &\text{and the escape is} \binom{a-1}{b-1}\binom{b-1}{i-1}\binom{n-1}{i-1}\binom{n-i}{m-1}.\end{split}
\end{equation}
 Putting \eqref{oi}, \eqref{pml} and \eqref{no} together, we get \eqref{rolj}.

 What is left for us to do is to prove \eqref{no}. Indeed, fix $1\le i\le b\wedge (n-m+1).$
On one hand, since the local time at $x$ is $a,$ the walk should visit $x$ for $a$ times, or in other words, it must jump away from $x$  to $x+1$ or $x-1$ for $a$ times. On the other hand, since the upcrossing time at $x$ is $b,$ the walk should jump from $x$ to $x+1$ for $b$ times and jump from $x$ to $x-1$ for $a-b$ times. Those jumps from $x$ to $x-1$ form the beginnings of those $a-b$ backward loops $BL_x,$ while those $b$ jumps from  $x$ to $x+1$ form the beginnings of those $i$ forward moves $F_{xy}$ and $b-i$ forward loops $FL_{xy}.$
But we know exactly that when standing at $x$ for the last time, it jumps from $x$ to $x+1$ and never returns to $x.$  Except for such a jump, we do not know exactly which $a-b$ jumps of the other $a-1$ jumps are from $x$ to $x-1.$  There are $\binom{a-1}{b-1}$ ways to choose $a-b$ jumps from those $a-1$ jumps to jump from $x$ to $x-1$ and form the beginnings of those $a-b$ backward loops $BL_x.$

Up to now, we already know which $b$ jumps (out of $a$ jumps) of the walk are from $x$ to $x+1.$ They form the beginnings of those $b-i$ forward loops $FL_{xy}$ and  $i$ forward moves $F_{xy}.$  As mentioned above,   the last jump of those $b$ jumps is from  $x$ to $x+1,$ which  forms the beginning of an upward move $F_{xy}.$ There are  $\binom{b-1}{i-1}$ ways to choose $b-i$ jumps from the left $b-1$ jumps  to form the beginning of those $b-i$ forward loops $FL_{xy}$ and the other $i-1$ jumps  to form the beginnings of the other $i-1$ forward moves $F_{xy}.$

We now know there are $b$ jumps that the walk takes from $x$ to $x+1$ and already know  which $b-i$ jumps of them form the beginnings of the forward loops $FL_{xy}$ (which end with a jump from $x+1$ to $x$) and which $i$ jumps of them form the beginnings of the forward moves $F_{xy}.$ But since the upcrossing time at $x$ is $b,$ except for the last time, every time the walk jumps forward from $x$ to $x+1,$ after a number of steps, it must return to $x.$
Thus, except for the last forward move $F_{xy}$ from $x$ to $y,$ the other $i-1$ forward moves $F_{xy}$ from $x$ to $y$ must be followed by a  number (possibly $0$) of forward loops $FL_{y}$ and/or a number(possibly $0$) of backward loops $BL_{yx},$ and then, with probability $1$ a backward move $F_{yx}$ from $y$ to $x.$ The walk jumps from $y$ to $y+1$ or $y-1$ for $n$ times. We also know that when the last time the walk standing at $y,$ it jumps from $y$ to $y+1$ and never returns to $y.$ There are $\binom{n-1}{i-1}$ possible ways to choose $i-1$ jumps from the other $n-1$ jumps to form the beginning of those backward moves $B_{yx}$ from $y$ to $x.$

Until now, we have fixed $i-1$ jumps from $y$ to $y-1,$ which are the beginnings of those $i-1$ backward moves $F_{yx}$ from $y$ to $x$ and one jump from $y$ to $y+1,$ which is the beginning of the escape $F_y$ from $y$ to $\infty.$ But the walk will leave $y$ for $n$ times. Therefore, except for the above already fixed $i$ jumps, the other $n-i$ jumps that the walk jumps away from $y$ should be the beginnings of those $m-1$ forward loops $FL_y$ and $n-i-(m-1)$ backward loops at $y.$ There are $\binom{n-i}{m-1}$ ways to choose $m-1$ jumps from those $n-i$ jumps to serve as the beginnings of those forward loops $FL_y$  at $y.$ So far, we have already determined the order of that escape and those loops and moves.

Summing up the above discussion, we get \eqref{no}. Thus, \eqref{rolj} is proved.

Finally, taking summation on both sides of \eqref{rolj} over $a$ and $n$ from $b$ to $\infty$ and $m$ to $\infty,$ respectively, and using the formula $
  \sum_{j=0}^\infty\binom{k+j}{j} t^{j}=\frac{1}{(1-t)^{k+1}},
$ by some tedious but elementary computation,  we get \eqref{jpu}.  The lemma is proved. \qed

\begin{remark}
  The above combinatorial proof of Lemma \ref{aaxy} related to loops, moves and escapes is  relatively short and easy to understand. The idea of such a proof comes from  an anonymous referee of an early unpublished version of our manuscript \cite{w23}, which contains only criterion for the finiteness of $C(a,*),$ and the limit distribution of $C(1,1).$  But we worry that some of the readers may doubt the rigorousness of such a combinatorial proof. So we also provide a path decomposition proof here, which is much longer.
\end{remark}

\noindent{\it Proof of Lemma \ref{aaxy}(Path decomposition scope).}
In this proof, for a $j$-dimensional  vector $\vec{\mathbf v},$ we stipulate that its $k$th coordinate is $v_k,$ $1\le k\le j.$   For two $j$-dimensional vectors $\vec{\mathbf w}$ and $\vec{\mathbf u}$ we write $\vec{\mathbf w}\prec \vec{\mathbf u}$ if $w_j<u_j$ and $w_k\le u_k,$ $1\le k\le j-1.$

We show only \eqref{rolj}, since \eqref{jpu} can be derived from \eqref{rolj}.
Denote again $\mathcal A=\{\xi(x)=a,\xi(x,\uparrow)=b,\xi(y)=n,\xi(y,\uparrow)=m\}.$
Let $\tau_0=0,$ and for $k\ge 1,$ let $$\tau_k=\inf\{n>\tau_{k-1}:X_{n-1}=x, X_{n}=x+1\}.$$
Then we have $$0=\tau_0\le \tau_1\le \tau_2\le ...\le \tau_b\le \tau_{b+1}\le \infty,$$ where for $k\ge1,$ $\tau_k<\tau_{k+1}$ if $\tau_{k+1}<\infty.$ Notice that on $\mathcal A,$ $\tau_{b}<\tau_{b+1}=\infty.$
For $1\le k\le b,$ let
\begin{align}
 V_k&=\#\{\tau_{k-1}\le n<\tau_k: X_n=x\},\label{dv}\\
U_k&=\#\{\tau_{k}\le n<\tau_{k+1}: X_n=y\}.\label{du}
\end{align}
Furthermore, let
\begin{align}
  W_b+1=\#\{\tau_{b}\le n<\tau_{b+1}: X_{n-1}=y,X_n=y+1\}\label{dw1}
\end{align} and
for $1\le k\le b-1,$ set
\begin{align}\label{dw2}
  W_k=\#\{\tau_{k}\le n<\tau_{k+1}: X_{n-1}=y,X_n=y+1\}.
\end{align}
Clearly, $\mathcal A$ occurs if and only if
\begin{align*}
  &V_k\ge 1,\ k=1,...,b,\ V_1+...+V_b=a,\\
  &U_b\ge1,\ U_k\ge 0,\ k=1,...,b-1,\ U_1+...+U_b=n,\\
  &0\le W_b<U_b,\ 0\le W_k\le U_k,\ k=0,...,b-1,\ W_1+...+W_b=m-1,
\end{align*} or equivalently,
 $$\vec{\mathbf{V}}\in S(a,b),\ \vec{\mathbf U}\in \tilde{S}(n,b),\ \vec{\mathbf W}\in S(m-1,b) \text{ and }  \vec{\mathbf W}\prec \vec{\mathbf U}.$$
Now fix $\vec{\mathbf v}\in S(a,b),$ $\vec{\mathbf u}\in \tilde{S}(n,b)$ and $\vec{\mathbf w}\in S(m-1,b).$ Denote
\begin{align}\label{davuw}
  \mathcal A(\vec{\mathbf v},\vec{\mathbf u}, \vec{\mathbf w})=\{\vec{\mathbf{V}}=\vec{\mathbf v},\vec{\mathbf{U}}=\vec{\mathbf u},\vec{\mathbf{W}}=\vec{\mathbf w}, \vec{\mathbf w}\prec \vec{\mathbf u}\}.
\end{align}
In order to compute the probability of the event $\mathcal A(\vec{\mathbf v},\vec{\mathbf u}, \vec{\mathbf w}),$ for $v,u,w\ge0,$ we define
\begin{align*}
  \mathcal B(x,v)&=\left\{\begin{array}{l}
    \text{starting from }  x, \text{ before it hits }x+1, \text{ the }\\
    \text{walk visits }x\ v\text{ times and then hits } x+1
  \end{array}
\right\};\\
\mathcal C(x,y,u,w)&=\left\{\begin{array}{l}
    \text{starting from }  x+1, \text{ before it hits }x, \text{ the walk }\\
   \text{visits } y\ u\text{ times, jumps from } y \text{ to } y+1\ w \text{ times}\\
     \text{and then, after a number of steps, it hits } x
  \end{array}
\right\};\\
\tilde{\mathcal C}(x,y,u,w)&=\left\{\begin{array}{l}
    \text{starting from }  x+1, \text{ the walk never returns back}\text{ to }x \\
\text{and visits } y\ u\text{ times, jumps from } y \text{ to } y+1\ w \text{ times}
  \end{array}
\right\}.
\end{align*}
 It is easily seen that
    \begin{align}\label{pbxv}
      \mathbb P(\mathcal B(x,v))=q_x^{v-1}p_x,\ v\ge 1.
    \end{align}
    Furthermore,
 we claim that \begin{align}\label{prc1}
  \mathbb P(\mathcal C(x,y,0,0))&=P_{x+1}(x,y,-) \end{align}
  and
  \begin{align}\label{prc}
 \mathbb P(\mathcal C(x,y,u,w))&=\binom{u-1}{w}P_{x+1}(x,y,+)(q_yP_{y-1}(x,y,+))^{u-w-1}\no\\
 &\times(p_yP_{y+1}(y,\infty,-))^wq_yP_{y-1}(x,y,-) \text{ for }0\le w< u.
\end{align}
Indeed, the event $\mathcal C(x,y,0,0)$ occurs if and only if starting from $x+1,$ the walk hits $x$ before it hits $y,$ which occurs with probability $P_{x+1}(x,y,-).$ Thus, we get \eqref{prc1}.  Next, fixing  $u>w\ge 0,$ we show \eqref{prc}. Notice that the event $\mathcal C(x,y,u,w)$ occurs if and only if the following events occur:

 (i) starting from $x+1,$  the walk hits $y$ before it hits $x;$

 (ii) restarting from $y,$ before it hits $x,$ it has $u$ jumps   away from  $y,$ of which the last one is from $y$ to $y-1$, and $w$ jumps are from $y$ to $y+1,$ and the other $u-w-1$ jumps are from $y$ to $y-1$.

 Clearly, the probability of the first event equals $P_{x+1}(x,y,+)$).
To see the probability of the second event, notice that standing at $y$ for the last time, the walk jumps to $y-1,$ and restarting from $y-1,$ it hits $x$ before it hits $y,$ with probability $q_yP_{y-1}(x,y,-).$ Next, we see what happens after the other $u-1$ times that the walk stands at $y.$ For $w$ times, it moves from $y$ to $y+1$ and returns back to $y,$ with probability $p_yP_{y+1}(y,\infty,-).$ And for $u-w-1$ times, it moves from $y$ to $y-1,$ and restarting from $y-1,$ it returns back to $y$ before it hits $x,$ with probability $q_yP_{y-1}(x,y,+).$ But there are $\binom{u-1}{w}$ possible ways to choose $w$ jumps from those $u-1$ jumps to move forward from $y$ to $y+1.$ Consequently, the probability of the second event equals
\begin{align*}
  \binom{u-1}{w}(q_yP_{y-1}(x,y,+))^{u-w-1}(p_yP_{y+1}(y,\infty,-))^wq_yP_{y-1}(x,y,-).
\end{align*}
Therefore, \eqref{prc} is true.

Putting \eqref{prc1} and \eqref{prc} together, we get
\begin{align}\label{pr}
  \mathbb P(\mathcal C(x,y,u,w))&=1_{\{u=w=0\}}P_{x+1}(x,y,-)\no\\
  &\quad\quad+ 1_{\{0\le w<u\}}\binom{u-1}{w}P_{x+1}(x,y,+)(q_yP_{y-1}(x,y,+))^{u-w-1}\no\\
  &\quad\quad\quad\quad\quad\quad\quad\quad\quad\times(p_yP_{y+1}(y,\infty,-))^wq_yP_{y-1}(x,y,-)
\end{align}

 Notice that if $\tilde{\mathcal C}(x,y,u,w)$ occur,   when standing at $y$ for the last time, the walk moves forward from $y$ to $y+1$ and never returns back to $y,$ with probability $p_yP_{y+1}(y,\infty,+).$ Thus similar to \eqref{prc}, we obtain
 \begin{align}\label{prtc}
 \mathbb P( \tilde{\mathcal C}(x,y,u,w))=&\binom{u-1}{w}P_{x+1}(x,y,+)(q_yP_{y-1}(x,y,+))^{u-w-1}\no\\
 &\times(p_yP_{y+1}(y,\infty,-))^wp_yP_{y+1}(y,\infty,+) \text{ for }0\le w<u.
\end{align}

 We are now ready to compute the probability of the event $\mathcal A(\vec{\mathbf v},\vec{\mathbf u}, \vec{\mathbf w})$ defined in \eqref{davuw}. Notice that $\vec{\mathbf w}\prec \vec{\mathbf u}.$ Thus, if $u_i=0,$ then $w_i=0,$ for $1\le i\le b-1.$

In view of \eqref{dv}-\eqref{dw2}, $\mathcal A(\vec{\mathbf v},\vec{\mathbf u}, \vec{\mathbf w})$ occurs if and only if the event
$$\bigcap_{k=1}^{b-1}\mathcal B(x,v_k)\mathcal C(x,y,u_k,w_k)\mathcal B(x,v_b)\tilde{\mathcal C}(x,y,u_b,w_b)$$
occurs. Thus, by the Markov property, on accounting of \eqref{pbxv}, \eqref{pr} and \eqref{prtc}, we have
\begin{align}
  \mathbb P(\mathcal A&(\vec{\mathbf v},\vec{\mathbf u}, \vec{\mathbf w}))=\mathbb P\z(\bigcap_{k=1}^{b-1}\mathcal B(x,v_k)\mathcal C(x,y,u_k,w_k)\mathcal B(x,v_b)\tilde{\mathcal C}(x,y,u_b,w_b)\y)\no\\
  &=\z(\prod_{k=1}^{b-1}P(\mathcal B(x,v_k))\mathbb P(C(x,y,u_k,w_k))\y)P(\mathcal B(x,v_b))\mathbb P(\tilde{\mathcal C}(x,y,u_b,w_b))\no\\
  &=\bigg(\prod_{k=1}^{b-1}q_x^{v_k-1}p_x\bigg(1_{\{u_k=0\}}P_{x+1}(x,y,-)\no\\
   &\quad\quad\quad\quad\quad\quad+1_{\{u_k\ne 0\}}\binom{u_k-1}{w_k}P_{x+1}(x,y,+)(p_yP_{y+1}(y,\infty,-))^{w_k}\no\\
   &\quad\quad\quad\quad\quad\quad\quad\quad\times(q_yP_{y-1}(x,y,+))^{u_k-w_k-1}q_yp_{y-1}(x,y,-)\bigg)\bigg)\no\\
   &\quad\quad\times \bigg(q_x^{v_b-1}p_x
   \binom{u_b-1}{w_b} P_{x+1}(x,y,+)\no\\
   &\quad\quad\times(p_yP_{y+1}(y,\infty,-))^{w_b}(q_yP_{y-1}(x,y,+))^{u_b-w_b- 1}p_yp_{y+1}(y,\infty,+)\bigg)\no\\
   &=q_x^{a-b}p_x^b\z(\prod_{k=1}^{b}\binom{u_k-1}{w_k}\y)(P_{x+1}(x,y,-))^{\sum_{k=1}^{b-1}1_{u_k=0}}\no\\
   &\quad\quad\times (P_{x+1}(x,y,+))^{\sum_{k=1}^{b}1_{u_k\ne0}}\z(p_yP_{y+1}(y,\infty,-)\y)^{\sum_{k=1}^{b-1}w_k1_{u_k\ne0}}\no\\
   &\quad\quad\times \z(q_yP_{y-1}(x,y,+)\y)^{\sum_{k=1}^{b-1}(u_k-w_k-1)1_{u_k\ne 0}}(q_yP_{y-1}(x,y,-))^{\sum_{k=1}^{b-1}1_{u_k\ne0}}\no\\
  &\quad\quad\times(p_yP_{y+1}(y,\infty,-))^{w_b}(q_yP_{y-1}(x,y,+))^{u_b-w_b-1}p_yp_{y+1}(y,\infty,+),\no
\end{align}
where and in what follows, we use the convention that $\binom{u_k-1}{w_k}=1$ if $w_k\ge u_k-1,$ $k=1,...,b.$

Since $\vec{\mathbf{v}}\in S(a,b),$ $\vec{\mathbf{u}}\in\tilde{S}(n,b),$  $\vec{\mathbf w}\in S(m-1,b),$ and $w_k\le  u_k$ for all $1\le k\le b,$ thus, we have $u_b\ge1,$ $w_b+\sum_{k=1}^{b-1}w_k1_{u_k\ne 0}=\sum_{k=1}^{b}w_b=m-1$ and $u_b+\sum_{k=1}^{b-1}u_b1_{u_b\ne0}=n.$ Therefore, we get
  \begin{align}
   \mathbb P(\mathcal A&(\vec{\mathbf v},\vec{\mathbf u}, \vec{\mathbf w}))\no\\
   &=q_x^{a-b}p_x^b\z(\prod_{k=1}^{b}\binom{u_k-1}{w_k}\y)(P_{x+1}(x,y,-))^{\sum_{k=1}^{b}1_{u_k=0}}\no\\
   &\quad\quad\times (P_{x+1}(x,y,+))^{\sum_{k=1}^{b}1_{u_k\ne0}}\z(p_yP_{y+1}(y,\infty,-)\y)^{m-1}\no\\
   &\quad\quad\times \z(q_yP_{y-1}(x,y,+)\y)^{n-(m-1)-\sum_{k=1}^{b}1_{u_k\ne0}}(q_yP_{y-1}(x,y,-))^{\sum_{k=1}^{b}1_{u_k\ne0}-1}\no\\
  &\quad\quad\times p_yp_{y+1}(y,\infty,+).\no
    \end{align}
    Consequently, taking \eqref{ep} and Lemma \ref{ns} into account, we conclude that
  \begin{align}
  \mathbb P(\xi&(x)=a,\xi(x,\uparrow)=b,\xi(y)=n,\xi(y,\uparrow)=m)\no\\
  &=\sum_{\vec{\mathbf v}\in S(a,b),\vec{\mathbf u}\in \tilde S(n,b),\vec{\mathbf w}\in S(m-1,b),\vec{\mathbf w}\prec \vec{\mathbf u} }\mathbb P(\mathcal A(\vec{\mathbf v},\vec{\mathbf u}, \vec{\mathbf w}))\no\\
   &=\sum_{\vec{\mathbf v}\in S(a,b),\vec{\mathbf u}\in \tilde S(n,b),\vec{\mathbf w}\in S(m-1,b),\vec{\mathbf w}\prec \vec{\mathbf u} }
   q_x^{a-b}p_x^b\z(\prod_{k=1}^{b}\binom{u_k-1}{w_k}\y)\frac{p_y}{D(y)}\no\\
  &\times \z(1-\frac{1}{D(x,y)}\y)^{\sum_{k=1}^{b}1_{u_k=0}} \z(\frac{1}{D(x,y)}\y)^{\sum_{k=1}^{b}1_{u_k\ne0}}\z(p_y\z(1-\frac{1}{D(y)}\y)\y)^{m-1}\no\\
  &\times \z(q_y\frac{D(x,y-1)}{D(x,y)}\y)^{n-(m-1)-\sum_{k=1}^{b}1_{u_k\ne0}}
   \z(q_y\z(1-\frac{D(x,y-1)}{D(x,y)}\y)\y)^{\sum_{k=1}^{b}1_{u_k\ne0}-1}\no\\
   &=\sum_{i=1}^{b\wedge (n-m+1)}\binom{a-1}{b-1}\sum_{\vec{\mathbf u}\in \tilde S_i(n,b),\vec{\mathbf w}\in S(m-1,b),\vec{\mathbf w}\prec \vec{\mathbf u} }
   q_x^{a-b}p_x^b\z(\prod_{k=1}^{b}\binom{u_k-1}{w_k}\y)\frac{p_y}{D(y)}\no\\
  &\times \z(1-\frac{1}{D(x,y)}\y)^{b-i} \z(\frac{1}{D(x,y)}\y)^{i}\z(p_y\z(1-\frac{1}{D(y)}\y)\y)^{m-1}\no\\
  &\times \z(q_y\frac{D(x,y-1)}{D(x,y)}\y)^{n-(m-1)-i}
   \z(q_y\z(1-\frac{D(x,y-1)}{D(x,y)}\y)\y)^{i-1}\no\\
   &=\sum_{i=1}^{b\wedge (n-m+1)}\binom{a-1}{b-1}\binom{b-1}{i-1}\sum_{\vec{\mathbf u}\in S(n,i),\vec{\mathbf w}\in S(m-1,i),\vec{\mathbf w}\prec \vec{\mathbf u} }
   q_x^{a-b}p_x^b\z(\prod_{k=1}^{i}\binom{u_k-1}{w_k}\y)\no\\
  &\times \frac{p_y}{D(y)}\z(1-\frac{1}{D(x,y)}\y)^{b-i} \z(\frac{1}{D(x,y)}\y)^{i}\z(p_y\z(1-\frac{1}{D(y)}\y)\y)^{m-1}\no\\
  &\times \z(q_y\frac{D(x,y-1)}{D(x,y)}\y)^{n-(m-1)-i}
   \z(q_y\z(1-\frac{D(x,y-1)}{D(x,y)}\y)\y)^{i-1}\no\\
    &=\binom{a-1}{b-1}\sum_{i=1}^{b\wedge (n-m+1)}\binom{b-1}{i-1}\binom{n-1}{i-1}\binom{n-i}{m-1}
   q_x^{a-b}p_x^b\frac{p_y}{D(y)}\no\\
  &\times \z(1-\frac{1}{D(x,y)}\y)^{b-i} \z(\frac{1}{D(x,y)}\y)^{i}\z(p_y\z(1-\frac{1}{D(y)}\y)\y)^{b-1}\no\\
  &\times \z(q_y\frac{D(x,y-1)}{D(x,y)}\y)^{a-(b-1)-i}
   \z(q_y\z(1-\frac{D(x,y-1)}{D(x,y)}\y)\y)^{i-1},\no
\end{align}
where for the last equality we use the fact that   $$\sum_{\vec{\mathbf w}\in S(m-1,i),\vec{\bf w}\prec \vec{\bf u}}\prod_{k=1}^i\binom{u_k-1}{w_k}=\binom{n-i}{m-1}, \text{ for each }  \vec{\bf u}\in S(n,i). $$
The lemma is proved. \qed

%%%%%%%%%%%%%%%%%%%%%%%%%%%%%%%%%%%%%%%%%%%%%%%%%%
\subsubsection{Proof of Proposition \ref{propc} for $C=C(a,b)$ or $C(*,a)$}
\proof Fix $a\ge b\ge1.$ Suppose $\rho_k\rightarrow 1$ as $k\rto.$  Then  we have $p_x\rightarrow 1/2,$ $q_x\rightarrow 1/2,$ $D(x)\rto $ as $x\rto$  and $p_k\vee p_k^{-1}\vee q_k\vee q_k^{-1}<c,$ $\rho_k\vee \rho_k^{-1}\le c,$ for all $k\ge1.$ Firstly, it follows directly from \eqref{dxa} and \eqref{dxb} that
\begin{align}\label{pxcu}
  \mathbb P(x\in C(a,b))<\binom{a-1}{b-1}\frac{1}{D(x)},\ \mathbb P(x\in C(*,a))<\frac{1}{D(x)}
\end{align}
and
\begin{align*}
  \lim_{x\rto}\mathbb P(x\in C(a,b))D(x)=\frac{1}{2^a}\binom{a-1}{b-1},\ \lim_{x\rto}\mathbb P(x\in C(*,a))D(x)=1.
\end{align*}
We thus get \eqref{lpc} for $C=C(a,b)$ or $C(*,a).$

Furthermore, it is straightforward to see from  \eqref{dxa}, \eqref{rolj} and \eqref{dxb}, \eqref{jpu} that
\begin{align*}
  \mathbb P(y\in C|x\in C)<\frac{cD(x)}{D(x,y)D(y)},\text{ for }y\ge x\ge 0, \ C=C(a,b) \text{ or }C(*,a),
\end{align*}
which proves \eqref{wca}. Putting \eqref{wca} and \eqref{pxcu} together, on accounting of \eqref{xm} and \eqref{csa}-\eqref{cabx}, we get \eqref{jkbw}.

Next, we show \eqref{pw}. For this purpose, on one hand,
 from \eqref{dxa} and \eqref{rolj},  we get
\begin{align}
&\frac{\mathbb P(x\in C(a,b),y\in C(a,b))}{\mathbb P(x\in C(a,b))\mathbb P(y\in C(a,b))}\no\\ &=\frac{D(x)}{D(x,y)}\z(1-\frac{1}{D(x,y)}\y)^{b-1}\z(\frac{D(x,y-1)}{D(x,y)}\y)^{a-b}\z(1-\frac{1}{D(x)}\y)^{1-b}\no\\
&+\frac{D(x)}{D(x,y)}\binom{a-1}{b-1}^{-1}\sum_{i=2}^{b\wedge (a-b+1)}\binom{b-1}{i-1}\binom{a-i}{b-1}\binom{a-1}{i-1}\no\\
&\quad\quad\times\z(\frac{1}{D(x,y)}\y)^{i-1}\z(1-\frac{1}{D(x,y)}\y)^{b-i}\z(\frac{D(x,y-1)}{D(x,y)}\y)^{a-b+1-i}\no\\
   &\quad\quad\times \z(1-\frac{D(x,y-1)}{D(x,y)}\y)^{i-1}\z(1-\frac{1}{D(x)}\y)^{1-b}\no\\
   & =:(A)+(B).\label{ab}
    \end{align}
On the other hand, taking \eqref{dxb} and \eqref{jpu} together, we infer that
    \begin{align}
&\frac{\mathbb P(x\in C(*,a),y\in C(*,a))}{\mathbb P(x\in C(*,a))\mathbb P(y\in C(*,a))}\no\\
  &=\frac{D(x)}{D(x,y)} \z(1-\frac{1}{D(x,y)}\y)^{a-1}  \z(\frac{p_{y}}{1-q_{y}\frac{D(x,y-1)}{D(x,y)}}\y)^{a} \z(1-\frac{1}{D(x)}\y)^{1-a}\no\\
   & +   \frac{D(x)}{D(x,y)} \sum_{i=2}^{a}\binom{a-1}{i-1}\binom{a+i-2}{i-1}\z(1-\frac{1}{D(x,y)}\y)^{a-i} \z(\frac{1}{D(x,y)}\y)^{i-1} \no\\
  &\quad\quad\times \z(\frac{p_{y}}{1-q_{y}\frac{D(x,y-1)}{D(x,y)}}\y)^{a}
   \z(\frac{q_{y}\z(1-\frac{D(x,y-1)}{D(x,y)}\y)}{1-q_{y}\frac{D(x,y-1)}{D(x,y)}}\y)^{i-1}\z(1-\frac{1}{D(x)}\y)^{1-a}\no\\
&=:(D)+(E). \label{de}
\end{align}

Fix $\ve>0$ and $\delta>0$. We claim that there exists $M_1>0$  such that   $(A)$, $(B)$ in \eqref{ab} and $(D),$ $(E)$ in  \eqref{de} satisfy
    \begin{gather}\label{be}
       \frac{D(x,y)}{D(x)}\times \max\{(B),(E)\}  \le \delta/2, \ x>M_1,\ y>x+M_1,\\
    \label{ae}
       1-\delta/2\le\frac{D(x,y)}{D(x)}\times (A) \le 1+\delta/2, \ x>M_1,\ y>x+M_1,
    \end{gather}
    and
    \begin{align}\label{dve}
       1-\delta/2\le\frac{D(x,y)}{D(x)}\times (D) \le 1+\delta/2,\  x>M_1,\ y>x+M_1.
    \end{align}
    Indeed, choose $\eta>0$ such that  \begin{gather}
      \binom{a-1}{b-1}^{-1}\sum_{i=2}^{b\wedge (a-b+1)}\binom{b-1}{i-1}\binom{a-i}{b-1}\binom{a-1}{i-1}\eta(1+\eta)^{b-1} < \delta/2,\label{eea}\\
      \sum_{i=2}^{a}\binom{a-1}{i-1}\binom{a+i-2}{i-1}\eta\z(1+\eta\y)^{2a-2}<\delta/2,\label{eed}\\
      (1+\eta)^{a-1}<1+\delta/2,\  (1-\eta)^{2a-1}>1-\delta/2.\label{eeb}
      \end{gather}
        Since $\lim_{k\rto}\rho_k=1,$  by Lemma \ref{edxy}, there exists $M_1>0$ such that
\begin{gather}
  \z(1-\frac{1}{D(x)}\y)^{-1}<1+\eta,\
  \z(1-\frac{1}{D(y)}\y)^{-1}<1+\eta,\label{dea} \\ \frac{1}{D(x,y)}<\eta,\ 1-\frac{D(x,y-1)}{D(x,y)}<\eta,\  \frac{q_{y}}{1-q_{y}\frac{D(x,y-1)}{D(x,y)}}<1+\eta,\label{ded}\\
  \frac{p_{y}}{1-q_{y}\frac{D(x,y-1)}{D(x,y)}}>1-\eta, \ x>M_1, \ y>x+M_1.
  \label{deb}
\end{gather}
Thus, on accounting of \eqref{eea} and \eqref{eed}, we have
\begin{align}
&\frac{D(x,y)}{D(x)}\times (B) \no\\
&\quad\quad\le \binom{a-1}{b-1}^{-1}\sum_{i=2}^{b\wedge (a-b+1)}\binom{b-1}{i-1}\binom{a-i}{b-1}\binom{a-1}{i-1}\eta^{2i-2}(1+\eta)^{b-1}<\frac{\delta}{2}, \no\\
&\frac{D(x,y)}{D(x)}\times (E)\le \sum_{i=2}^{a}\binom{a-1}{i-1}\binom{a+i-2}{i-1}\eta^{2i-2}\z(1+\eta\y)^{2a-2}<\frac{\delta}{2},\no
    \end{align}
for all $x>M_1,y>x+M_1.$ We thus get \eqref{be}. To prove \eqref{ae} and \eqref{dve}, it is easy to see from \eqref{dea}-\eqref{deb} that
\begin{gather}\label{bbc}
 (1-\eta)^{a-1}\le\frac{D(x,y)}{D(x)}\times (A) \le(1+\eta)^{b-1}, \ x>M_1,\ y>x+M_1,\\
\label{bbd} (1-\eta)^{2a-1}\le\frac{D(x,y)}{D(x)}\times (D) \le(1+\eta)^{a-1},\ x>M_1,\ y>x+M_1.
\end{gather}
Consequently, putting \eqref{eeb}, \eqref{bbc} and \eqref{bbd}  together, we get \eqref{ae} and \eqref{dve}.
Finally, substituting \eqref{be}-\eqref{ae} into \eqref{ab} and \eqref{de}, we get
\begin{align}\label{pwd}
1-\delta\le \frac{D(x,y)}{D(x)}\frac{\mathbb P(x\in C,y\in C)}{\mathbb P(x\in C)\mathbb P(y\in C)}
  \le1+\delta,\end{align}
   for $ x>M_1,y>x+M_1,$  $ C=C(a,b)$  or $C(*,a).$
   But by \eqref{lpc},  there exists a number $M_2>0$ such that
\begin{align}\label{dkl}
  (\lambda_C-\delta)/D(k)\le \mathbb P(k\in C) \le (\lambda_C+\delta)/D(k),
\end{align}
for $k> M_2,$ $C=C(a,b)$ or $C(*,a).$  Therefore, letting $N=M_1\vee M_2,$ from \eqref{pwd}, we obtain
    \begin{align}
\frac{(\lambda_C-\delta)(1-\delta)D(x)}{D(x,y)D(y)}\le \mathbb P(y\in C|x\in C)\le\frac{(\lambda_C+\delta)(1+\delta)D(x)}{D(x,y)D(y)},  \label{wcdd}
  \end{align}
for $x>N,$ $y>x+ N,$ $C=C(a,b)$ or $C(*,a).$
Taking \eqref{dkl} and \eqref{wcdd} together,  thanks to \eqref{xm} and \eqref{csa}-\eqref{cabx}, we conclude that
\begin{align}\label{jawd}
\frac{((\lambda_C-\delta)(1-\delta))^k}{D(j_k)\prod_{i=1}^{k-1}D(j_i,j_{i+1})} &\le \mathbb P(j_1\in C,...,j_k\in C)\le \frac{((\lambda_C+\delta)(1+\delta))^k}{D(j_k)\prod_{i=1}^{k-1}D(j_i,j_{i+1})},
\end{align}
 for $0=j_0<j_1<...<j_k,\ j_i-j_{i-1}>N,\ i=1,...,k,$  $C=C(a,b)$ or $C(*,a).$
  Choosing $\delta$ small enough such that $(\lambda_C-\delta)(1-\delta)>\lambda_C-\ve$ and
$(\lambda_C+\delta)(1+\delta)<\lambda_C+\ve,$ from \eqref{pwd}, \eqref{wcdd} and \eqref{jawd}, we get  \eqref{pw}, \eqref{wc} and \eqref{jkaw}.

Finally, by multiplying $\mathbb P(x\in C),$ $C=C(a,b)$ or $C(*,a)$ on both sides of  \eqref{ab} or \eqref{de}, respectively, and discarding the second term on the righthand side, owing to  Lemma \ref{lxa}, we have
\begin{align}
  \mathbb P(x&\in C(a,b)|y\in C(a,b))\ge (A)\times \mathbb P(x\in C(a,b))\no\\
  &=\binom{a-1}{b-1}p_x^bq_x^{a-b}\frac{1}{D(x,y)}\z(1-\frac{1}{D(x,y)}\y)^{b-1}\z(\frac{D(x,y-1)}{D(x,y)}\y)^{a-b}\no\\
  &\ge \binom{a-1}{b-1}\frac{c}{D(x,y)} \z(1-\frac{1}{1+\rho_{x+1}}\y)^{b-1}\z(\frac{1}{1+\rho_{y-1}}\y)^{a-b}\no\\
  &\ge \frac{c}{D(x,y)}, \ y\ge x\ge0,\no
\end{align}
and \begin{align}
  \mathbb P(x&\in C(*,a)|y\in C(*,a))\ge (D)\times \mathbb P(x\in C(*,a))\no\\
   &=\frac{1}{D(x,y)} \z(1-\frac{1}{D(x,y)}\y)^{a-1}  \z(\frac{p_{y}}{1-q_{y}\frac{D(x,y-1)}{D(x,y)}}\y)^{a} \no\\
  &\ge \frac{1}{D(x,y)}\z(1-\frac{1}{1+\rho_{x+1}}\y)^{a-1}p_y^a\ge \frac{c}{D(x,y)},\ y\ge x\ge0.\no
\end{align}
That is, \eqref{pdl} is true for $C=C(a,b)$ or $C(*,a).$
 We thus finish the proof of Proposition \ref{propc}  for $C=C(a,b)$ or $C(*,a).$ \qed
\section{Proof of Theorem \ref{main}}\label{s4}
%%%%%%%%%%%%%%%%%%%%%%%%%%%%%%%%%%%%%%%%%%%%%%%%%%%%%%%%%%%%

Based on Lemma \ref{lemm} and Proposition \ref{propc}, we give the proof of Theorem \ref{main} in this  section.

\proof  Let $\Gamma_n=\{n\in C(a,b)\},$ or $\{n\in C_w\},$ or $\{n\in C(*,a)\},n\ge 1.$  Let $D(m)$ and $D(m,n)$ be as in \eqref{ddmn} and \eqref{dm}. Suppose that $\rho_k$ is increasing in $k>N_0$ for some $N_0>0$ and $\rho_k\rightarrow 1$ as $k\rto.$
Then clearly, we have
\begin{gather}
  D(n)\ge D(m),\ n>m>N_0,\label{dic}\\
  D(m,n)\le c(n-m),\ n>m\ge0, \label{dmnu}
\end{gather}
and from \eqref{lpc}, we see that \eqref{dng} holds.

Suppose now $\sum_{n=2}^\infty\frac{1}{D(n)\log n}<\infty.$ If $D(0)=\infty,$ then by Proposition \ref{crt}, the chain $X$ is recurrent almost surely. Therefore, $|C(a,b)|=|C(*,a)|=|C_w|=0$ almost surely, for $a\ge b\ge1.$ Assume next $D(0)<\infty.$  From Lemma \ref{lemm}, we see that \eqref{cki} is fulfilled. Putting \eqref{pdl} and \eqref{dmnu} together, we deduce that
\eqref{dxyc} is fulfilled.  Therefore, applying part (i) of Theorem \ref{fis}, we conclude that
\begin{align}
  |C(a,b)|\vee |C(*,a)|\vee |C_w|<\infty,\ a\ge b\ge1,\label{abs}
\end{align} almost surely.
Notice that for each $A\subseteq \{a,a+1,...\},$  $B\subset \{1,2,...,a\},$ we must have $|C(A,a)|<C(*,a)$ and $|C(a,B)|=\sum_{b\in B}|C(a,b)|.$ Consequently, we infer from \eqref{abs} that $|C(A,a)|<\infty$  and $|C(a,B)|<\infty$ almost surely for each $a\in \{1,2,...\}$ $A\subseteq \{a,a+1,...\},$ and $B\subset\{1,2,...,a\}.$ The convergent part of Theorem \ref{main} is proved.

Next, we prove the divergent part. To this end, suppose that there exist $n_0>0$ and $\delta>0$ such that  $D(n)\le \delta n\log n$ for all $n\ge n_0,$  and  $\sum_{n=2}^\infty\frac{1}{D(n)\log n}=\infty.$ From \eqref{dmn}, we see that \eqref{dmna} is fulfilled.
Moreover, from \eqref{dic}, we know that $D(n)$ is increasing in $n>N_0.$ Finally, from \eqref{pw}, we infer that \eqref{ppt} holds.
Thus, it follows from part (ii) of Theorem \ref{fis} that
\begin{align}
  |C(a,b)|=\infty, \ |C_w|=\infty, \text{ for } a\ge b\ge1,  \label{last}
\end{align}
almost surely.
Now, consider $a\ge1,$  $\phi\ne A\subseteq \{a,a+1,...\}$ and $\phi\ne B\subseteq \{1,2,...,a\}.$ Obviously, from \eqref{last}, we obtain
$|C(A,a)|=\sum_{x\in A}|C(x,a)|=\infty$ and $|C(a,B)|= \sum_{b\in B}|C(a,b)|=\infty,$ almost surely.  The divergent part of Theorem \ref{main} is proved. \qed

\section{Expectations-Proof of Proposition \ref{ec}}\label{s5}
In this section, letting the perturbations $r_n,n\ge1$ be as in \eqref{dr}, we estimate the expectations of the cardinalities  of the sets $C(a,b)\cap [1,n]$ and $C_w\cap [1,n].$ For this purpose, we first give the proof of Lemma \ref{dnl}.
\subsection{Proof of Lemma \ref{dnl}}

\proof It is shown in Cs\'aki et al. \cite{cfrc} (see page 637 therein) that $c_1n(\log\log n)^\beta<D(n)<c_2n(\log\log n)^\beta$ for some constants $0<c_1<c_2<\infty.$ The proof of Lemma \ref{dnl} is a refinement of such a result. We need the following lemma.
\begin{lemma}\label{lfr}
Suppose that $f(x),x\ge0$ is a nonnegative function,  $f(x)$ is decreasing in $x>M$ for some $M>0,$ and $\sum_{k=1}^{\infty}(f(k))^2<\infty.$ Set $\rho_n=1-f(n)+O((f(n))^2).$  Then, \begin{align}\label{cl8}
  \rho_1\cdots\rho_n\sim c \exp\z\{-\int_{1}^{n}f(x)dx\y\} \text{ as }n\rto.
\end{align}
\end{lemma}
\proof
Under the given conditions, by checking carefully, we can show that the series
$$\sum_{j=1}^\infty\z(\log \rho_j+\int_{j-1}^jf(x)dx\y)$$ is convergent. Thus, \eqref{cl8} is true. \qed

Now we are ready to finish the proof of Lemma \ref{dnl}.
Notice that $\rho_n=1-4r_n+O(r_n^2)$ as $n\rto.$
Letting $f(x)=\frac{1}{x}+\frac{1}{x(\log\log x)^\beta},\ x\ge 3,$ then we have $\rho_n=1-f(n)+O((f(n))^2),$ as $n\rto.$
 Thus, applying Lemma \ref{lfr} we get
\begin{align}\label{rl}
  \rho_1\cdots\rho_n\sim \frac{ c}{n}\exp\z\{-\int_{3}^{n}\frac{1}{x(\log\log x)^\beta}dx\y\} \text{ as }n\rto.
\end{align}
Set $B_n=:\sum_{i=n}^{\infty}\rho_1\cdots\rho_i$ for $n\ge1.$
Fix  $\ve>0.$ By \eqref{rl}  there exists $M_1>0$ such that
\begin{align}\label{re}
  (c-\ve)&\frac{ 1}{n}\exp\z\{-\int_{3}^{n}\frac{1}{x(\log\log x)^\beta}dx\y\}\no\\
  &\le\rho_1\cdots \rho_n\le (c+\ve)\frac{ 1}{n}\exp\z\{-\int_{3}^{n}\frac{1}{x(\log\log x)^\beta}dx\y\}, n>M_1.
\end{align}
Thus, we have
\begin{align}\label{bg}
  (c-\ve)&\sum_{i=n}^\infty\frac{ 1}{i}\exp\z\{-\int_{3}^{i}\frac{1}{x(\log\log x)^\beta}dx\y\}\no\\
  &\le B_n \le (c+\ve)\sum_{i=n}^\infty\frac{ 1}{i}\exp\z\{-\int_{3}^{i}\frac{1}{x(\log\log x)^\beta}dx\y\}, n>M_1.
\end{align}
Let $g(y)=\frac{1}{y}\exp\z\{-\int_{3}^{y}\frac{1}{x(\log\log x)^\beta}dx\y\},\ y\ge3.$ Since $g(y)$ is decreasing in $y\ge3,$ then we have $0<g(l)\le \sum_{j=n}^lg(j)-\int_{n}^{l}g(y)dy\le g(n)$ for $l>n\ge3.$ Thus, we deduce that
\begin{align}\no
  1\le \frac{\sum_{j=n}^\infty g(j)}{\int_{n}^{\infty}g(y)dy}\le 1+\frac{g(n)}{\int_{n}^{\infty}g(y)dy},\ n\ge3.
\end{align}
But  by  L'Hospital's rule,  $\lim_{n\rto}\frac{g(n)}{\int_{n}^{\infty}g(y)dy}=0.$ Therefore, there exists a number $M_2>0$ such that
\begin{align}\label{gi}
  1\le \frac{\sum_{j=n}^\infty g(j)}{\int_{n}^{\infty}g(y)dy}\le 1+\ve,\ n>M_2.
\end{align}
Substituting \eqref{gi} into \eqref{bg}, we infer that for $n>M_1\vee M_2,$
\begin{align*}
  (c-\ve)&\int_{n}^\infty\frac{ 1}{y}\exp\z\{-\int_{3}^{y}\frac{1}{x(\log\log x)^\beta}dx\y\}dy\no\\
  &\le B_n \le (c+\ve)(1+\ve)\int_{n}^\infty\frac{ 1}{y}\exp\z\{-\int_{3}^{y}\frac{1}{x(\log\log x)^\beta}dx\y\}dy.
\end{align*}
Since
\begin{align}\no
  \int_{n}^\infty\frac{ 1}{y}\exp\z\{-\int_{3}^{y}\frac{1}{x(\log\log x)^\beta}dx\y\}dy\sim (\log\log n)^\beta\exp\z\{\int_3^n\frac{1}{x(\log\log x)^\beta}dx \y\}
\end{align} as $n\rto$ (see Cs\'aki et al. \cite{cfrc}, page 637), there is $M_3>M_1\vee M_2$ such that  for $n>M_3$

\begin{align}\label{bup}
  (c-\ve)(1-\ve)&(\log\log n)^\beta\exp\z\{\int_3^n\frac{1}{x(\log\log x)^\beta}dx \y\}\no\\
  &\le B_n \le (c+\ve)(1+\ve)^2(\log\log n)^\beta\exp\z\{\int_3^n\frac{1}{x(\log\log x)^\beta}dx \y\}.
\end{align}
Noticing that $D(n)=\frac{B_n}{\rho_1\cdots\rho_n,}$ then from  \eqref{re} and \eqref{bup}, we get
\begin{align}\no
 \frac{(c-\ve)(1-\ve)}{c+\ve}n(\log\log n)^\beta<D(n)< \frac{(c+\ve)(1+\ve)^2}{c-\ve}n(\log\log n)^\beta,\ n >M_3.
\end{align}
We thus come to the conclusion that
\begin{align}\no
 D(n)\sim n(\log\log n)^\beta \text{ as }n\rto.
\end{align}
Lemma \ref{dnl} is proved. \qed
\subsection{Proof of Proposition \ref{ec}}
\proof Fix integers $a\ge b\ge 1.$
For $k\ge1,$ we set $\eta_k=1_{\{k\in C(a,b)\}},$ $\varphi_k=1_{\{k\in C_w\}}.$  Then we have $$|C(a,b)\cap [1,n]|=\sum_{k=1}^n \eta_k \text{ and } |C_w\cap [1,n]|=\sum_{k=1}^n \varphi_k.$$ It follows from Lemma \ref{lxa} that
\begin{align}
  \mathbb P(\eta_k=1)&=\mathbb P(k\in C(a,b))=
  \mathbb P(\xi(k)=a,\xi(k,\uparrow)=b)\no\\
  &=\binom{a-1}{b-1}\frac{p_k^bq_k^{a-b}}{D(k)}\z(1-\frac{1}{D(k)}\y)^{b-1},\no\\
  \mathbb P(\varphi_k=1)&=\mathbb P(k\in C_w)=\frac{1}{p_kD(k-1)}.\no
\end{align}
 Noting that $p_k\rightarrow 1/2,$ as $k\rto,$ thus applying Lemma \ref{dnl}, we see that \begin{align}
  &\mathbb P(\eta_k=1)\sim \binom{a-1}{b-1}\frac{1}{2^ak(\log\log k)^\beta} \text{ as }k\rto,\no\\
  &\mathbb P(\varphi_k=1)\sim \frac{ 2}{k(\log\log k)^\beta}\text{ as }k\rto.\no
\end{align}
Consequently, for $\ve>0,$ there exists a number $M_1>0$ such that
\begin{align}\no
 \binom{a-1}{b-1}\frac{1-\ve}{2^ak(\log\log k)^\beta}\le  &\mathbb P(\eta_k=1)\le \binom{a-1}{b-1}\frac{1+\ve}{2^ak(\log\log k)^\beta},\ k\ge M_1,\\
 \frac{2(1-\ve)}{k(\log\log k)^\beta}\le &\mathbb P(\varphi_k=1)\le \frac{2(1+\ve)}{k(\log\log k)^\beta},\ k\ge M_1.\no
\end{align}
But \begin{align*}
  \sum_{k=M_1}^n &\frac{1}{k(\log\log k)^\beta}\sim \int_{M_1}^n\frac{dx }{x(\log\log x)^{\beta}}\sim \frac{\log n}{(\log\log n)^{\beta}},\ n\rto.
  \end{align*}
Therefore, there exists a number $M_2>M_1$ such that for $n>M_2,$
\begin{align}
\binom{a-1}{b-1}\frac{(1-\ve)^2\log n}{2^a(\log\log n)^{\beta}}\le &\sum_{k=M_2}^n \mathbb P(\eta_k=1)\le \binom{a-1}{b-1}\frac{(1+\ve)^2\log n}{2^a(\log\log n)^{\beta}},\label{se}\\
\frac{2(1-\ve)^2\log n}{(\log\log n)^{\beta}}\le &\sum_{k=M_2}^n \mathbb P(\varphi_k=1)\le \frac{2(1+\ve)^2\log n}{(\log\log n)^{\beta}}.\label{svp}
\end{align}
As a result, we deduce  from \eqref{se} that
\begin{align}
  (1-\ve)^2&\le  \varliminf_{n\rto}\frac{2^a(\log\log n)^{\beta}}{\binom{a-1}{b-1}\log n}\mathbb E|C(a,b)\cap[1,n]|\no\\
  &\le \varlimsup_{n\rto}\frac{2^a(\log\log n)^{\beta}}{\binom{a-1}{b-1}\log n}\mathbb E|C(a,b)\cap[1,n]|\le (1+\ve)^2.\no
\end{align}
Since $\ve$ is arbitrary, letting $\ve\rightarrow0,$ we get \eqref{ecab}.
Similarly,  from \eqref{svp} we get \eqref{ecw}.
Proposition \ref{ec} is proved. \qed
\section{Limit distributions-Proof of Theorem \ref{cs}}\label{s6}
This section is devoted to proving Theorem \ref{cs}. The proof is based on the moment method. To start with, we give an auxiliary lemma to estimate $D(i)$ and $D(i,j).$

\subsection{An auxiliary lemma}
\begin{lemma}\label{dij} Set $r_1=1/4,$ and for $n\ge 2,$ let $r_n=\frac{1}{2n}.$ Suppose that $p_n=1/2+r_n,\ n\ge1.$ Then there exists a number $i_0>0$ such that
 \begin{align}\label{dup}
  (1-\varepsilon)\frac{i(j-i)}{j}\le D(i&,j)\le (1+\varepsilon)\frac{i(j-i)}{j},\ j>i\ge i_0,\\
\label{dupi}
  (1-\varepsilon) i\le D(&i)\le (1+\varepsilon)i,\ i\ge i_0.
\end{align}
\end{lemma}
\proof
Notice that $$\rho_n\sim 1- 4r_n+O(r_n^2)=1-\frac{2}{n}+O(n^{-2}), \ n\rto.$$
Then for $j>i\ge1,$ choosing  $\theta_n\in(0\wedge(\rho_n-1),0\vee(\rho_n-1)),\ i\le n\le j$ properly, we have
\begin{align*}
  \rho_i\cdots \rho_j&=e^{\sum_{n=i}^j\log \rho_n}=e^{-\sum_{n=i}^j\frac{(\rho_n-1)^2}{2(1+\theta_n)^2}}e^{\sum_{n=i}^j(\rho_n-1)}\\
  &\ge e^{-\sum_{n=i}^\infty\frac{cn^{-2}}{2(1+\theta_n)2}}e^{-2\sum_{n=i}^j\frac{1}{n}}e^{-\sum_{n=i}^\infty cn^{-2}}\\
  &\ge e^{-\sum_{n=i}^\infty\frac{cn^{-2}}{2(1+\theta_n)2}}e^{-2\int_{i}^{j+1}\frac{1}{x}dx}e^{-\sum_{n=i}^\infty cn^{-2}}\\
  &=\z(\frac{i}{j+1}\y)^2 e^{-\sum_{n=i}^\infty\frac{cn^{-2}}{2(1+\theta_n)2}}e^{-\sum_{n=i}^\infty cn^{-2}}.
\end{align*}
Fix $\varepsilon>0.$ We can find a number $i_1$ such that
$\frac{i}{i-1}\frac{j}{j+1}e^{-\sum_{n=i}^\infty\frac{cn^{-2}}{2(1+\theta_n)2}}e^{-\sum_{n=i}^\infty cn^{-2}}>(1-\varepsilon)^{1/2}$ for all $i\ge i_1.$ Thus we have \begin{align}
  \rho_i\cdots \rho_j\ge (1-\varepsilon)^{1/2}\z(\frac{i-1}{j}\y)^2, \text{ for all }j\ge i\ge i_1.\no
\end{align}
Similarly, we can find a number $i_2>i_1$ such that
\begin{align}
  \rho_i\cdots \rho_j\le (1+\varepsilon)^{1/2}\z(\frac{i-1}{j}\y)^2, \text{ for all }j\ge i\ge i_2.\no
\end{align}
Consequently, we get
\begin{align}
 (1-\varepsilon)^{1/2}\z(\frac{i-1}{j}\y)^2\le \rho_i\cdots \rho_j\le (1+\varepsilon)^{1/2}\z(\frac{i-1}{j}\y)^2, \text{ for all }j\ge i\ge i_2.\no
\end{align}
As a result, we obtain
\begin{align}
  (1-\varepsilon)^{1/2}\sum_{n=i}^{j-1}\z(\frac{i}{n}\y)^2\le D(i,j)\le (1+\varepsilon)^{1/2}\sum_{n=i}^{j-1}\z(\frac{i}{n}\y)^2,\ j>i\ge i_2.\no
\end{align}
Notice that for some number $i_0>i_2,$ we have
\begin{align}
 (1-\ve)^{1/2} \frac{i(j-i)}{j}\le  \sum_{n=i}^{j-1}\z(\frac{i}{n}\y)^2\le (1+\ve)^{1/2} \frac{i(j-i)}{j},\ j> i\ge i_0.\no
\end{align}
 Therefore, we get
 \begin{align}
  (1-\varepsilon)\frac{i(j-i)}{j}\le D(i,j)\le (1+\varepsilon)\frac{i(j-i)}{j},\ j>i\ge i_0\no
\end{align} which proves \eqref{dup}.
 Letting $j\rto$ in \eqref{dup}, we get \eqref{dupi}.
  Lemma \ref{dij} is proved. \qed
\subsection{Convergence of the moments}
An important step to prove Theorem \ref{cs} is to show the convergence of $\frac{\mathbb E(|C\cap [1,n]|^k)}{(\log n)^k},$ for $k\ge 1,$ $C=C(a,b),$ $C(*,a)$ or $C_w,$  as $n\rto.$ We have the following lemma.
\begin{lemma}\label{com} Fix integers $a\ge b\ge1.$ Let $r_1=1/4 $ and  $r_n=\frac{1}{2n}$  for $n\ge2.$  Set $p_n=\frac{1}{2}+r_n$ for $n\ge1.$ Then we have
 \begin{align*}
   &\lim_{n\rto}\frac{ \mathbb E(|C\cap [1,n]|^k)}{(\log n)^k}=\left\{\begin{array}{cl}
     \frac{k!}{2^{ak}}\binom{a-1}{b-1}^k & \text{if }\ C=C(a,b),\\
  k!&\text{if }\ C=C(*,a),\\
  2^kk!&\text{if }\ C=C_w,
   \end{array}\right. \text{ for }\ k\ge1.
 \end{align*}
  \end{lemma}
\proof Consider $C= C(a,b),C_w$ or $C(*,a).$ Let $\lambda_C$ be the one in \eqref{lpc}. For $k\ge1,$  set $\eta_k=1_{\{k\in C\}},$ clearly we have $|C\cap [1,n]|=\sum_{k=1}^n \eta_k. $

Let $S(a,j)$ be the one defined in \eqref{saj}.  Then
\begin{align}
  \mathbb \mathbb E|C\cap [1,n]|^k&=\mathbb E\z(\z(\sum_{j=1}^n\eta_j\y)^k\y)=\sum_{0< j_1,j_2,...,j_k\le n}\mathbb E\z(\eta_{j_1}\eta_{j_2}\cdots\eta_{j_k}\y)\no\\
  &=\sum_{m=1}^k\sum_{\begin{subarray}{c}l_1+\dots+l_m=k,\\
  l_i\ge1,i=1,...,m
  \end{subarray}}m!\binom{k}{m}\sum_{0< j_1<...<j_m\le n}\mathbb E\z(\eta_{j_1}^{l_1}\cdots\eta_{j_m}^{l_m}\y)\no\\
  &=\sum_{m=1}^k\sum_{(l_1,...,l_m)\in S(k,m)}m!\binom{k}{m}\sum_{0< j_1<...<j_m\le n}\mathbb E\z(\eta_{j_1}^{l_1}\cdots\eta_{j_m}^{l_m}\y).\no
\end{align}
Since the values of $\eta_j,j=1,...,n$ are either $0$ or $1,$ then taking Lemma \ref{ns}  into account, we have
\begin{align}
  &\mathbb E|C\cap [1,n]|^k=\sum_{m=1}^k\binom{k-1}{m-1}m!\binom{k}{m}\sum_{0< j_1<...<j_m\le n}\mathbb E(\eta_{j_1}\cdots\eta_{j_m})\no\\
  &=\sum_{m=1}^k\binom{k-1}{m-1}m!\binom{k}{m}\sum_{0< j_1<...<j_m\le n}\mathbb P(j_1\in C,..,j_m\in C)\no\\
  &=:\sum_{m=1}^k\binom{k-1}{m-1}m!\binom{k}{m}G(n,m).\label{esn}
\end{align}
Fix $\ve>0.$ Let $N$ and $i_0$ be as in Proposition \ref{propc} and Lemma \ref{dij}, respectively. Set $n_2=N\vee i_0$ and
denote temporarily \begin{align*}
    &A=\{(j_1,...,j_m)\mid 0< j_1<....<j_m\le n, j_{i+1}-j_i>n_2, \forall 0\le i\le m-1\},\\
    &B=\{(j_1,...,j_m)\mid 0< j_1<....<j_m\le n, (j_1,...,j_m)\notin A\},\no\\
    &B_i=\{(j_1,...,j_m)\mid j_1<...<j_{i}\le n_2<j_{i+1}<...<j_m\le n\},i=1,...,m-1,\\
  &B_m=\{(j_1,...,j_m)| n_2<j_1<..<j_m, \{1\le s\le m-1|j_{s+1}-j_s<n_2\}\ne \phi\},    \end{align*} where and in the remainder of this proof we set $j_0=0.$
Obviously, we have $B=\bigcup_{i=1}^m B_i.$
Therefore, we can write
\begin{align}\label{gab}
  G(n,m)&=\sum_{(j_1,...,j_m)\in A}+\sum_{i=1}^m\sum_{(j_1,...,j_m)\in B_i}\mathbb P(j_1\in C,..,j_m\in C)\no\\
  &=:G_A(n,m)+\sum_{i=1}^mG_{B_i}(n,m).
  \end{align}
  Consider first the term $G_A(n,m).$ From   \eqref{jkaw}, \eqref{dup} and \eqref{dupi}, we get
  \begin{align*}
   G_A(n,m)&=\sum_{(j_1,...,j_m)\in A}\mathbb P(j_1\in C,...,j_m\in C)\\
   &\le(\lambda_C+\ve)^m(1+\ve)^m\sum_{(j_1,...,j_m)\in A} \frac{1}{\z(\prod_{i=1}^{m-1}D(j_i,j_{i+1})\y)D(j_m)}\no\\
  &\le(\lambda_C+\ve)^m(1+\ve)^m \sum_{(j_1,...,j_m)\in A} \frac{1}{j_1(j_2-j_1)\cdots (j_m-j_{m-1})}.\no
\end{align*}
  Applying Lemma \ref{lems},  we obtain
  \begin{align}\label{gau}
    \varlimsup_{n\rto} \frac{G_A(n,m)}{(\log n)^m}\le (\lambda_C+\ve)^m(1+\ve)^m.
  \end{align}
   Similarly, we have
     \begin{align}\label{gal}
    \varliminf_{n\rto} \frac{G_A(n,m)}{(\log n)^m}\ge (\lambda_C-\ve)^m(1-\ve)^m.
  \end{align}
  Next, we claim that
  \begin{align}
    \lim_{n\rto}\frac{G_{B_i}(n,m)}{(\log n)^m}=0,i=1,...,m.\label{gbi}
  \end{align}
  Indeed, for each $1\le i\le m,$ from  \eqref{jkbw}, we see that
  \begin{align}\label{gu}
    G_{B_i}(n,m)\le \sum_{(j_1,...,j_m)\in B_i} \frac{c}{\z(\prod_{i=1}^{m-1}D(j_i,j_{i+1})\y)D(j_m)}.
  \end{align}
  Then, in view of \eqref{lsb}, \eqref{dup} and \eqref{dupi}, we have
\begin{align}
  \varlimsup_{n\rto}\frac{G_{B_m}(n,m)}{(\log n)^m}\le \varlimsup_{n\rto}\frac{c}{(\log n)^m}\sum_{(j_1,...,j_m)\in B_m}\frac{1}{j_1(j_2-j_1)\cdots (j_m-j_{m-1})}=0.\no
\end{align}
 For $1\le i\le m-1,$  from  \eqref{dup}, \eqref{dupi} and \eqref{gu}, we deduce that
  \begin{align}
 G_{B_i}(n,m)&\le c\sum_{\begin{subarray}{c}0\le j_1<...<j_{i}\le n_2,\\
  n\ge j_{m}>...>j_{i+1}> n_2\end{subarray}}\frac{1}{\z(\prod_{s=1}^{m-1}D(j_s,j_{s+1})\y)D(j_m)},\no\\
  &\le c \sum_{0<j_1<...<j_{i}\le n_2}\frac{1}{\prod_{s=1}^{i}D(j_s,j_{s+1})}\no\\
  &\quad\quad\quad\quad\times\sum_{n\ge j_{m}>...>j_{i+1}> n_2}\frac{1}{j_{m}\prod_{s=i+1}^{m-1}(j_{s+1}-j_s)}\no\\
  &\le c\binom{n_2}{i}\max_{0<j_1<...<j_i\le n_2}\frac{1}{\prod_{s=1}^{i}D(j_s,j_{s+1})}\no\\
   &\quad\quad\quad\quad\times\sum_{n\ge j_{m-i}>...>j_{1}> n_2}\frac{1}{j_{m-i}\prod_{s=1}^{m-i-1}(j_{s+1}-j_s)}.\no
\end{align}
But it follows from Lemma \ref{lems} that
\begin{align}
  \lim_{n\rto}\frac{1}{(\log n)^{m-i}}\sum_{n\ge j_{m-i}>...>j_{1}> n_2}\frac{1}{j_{m-i}\prod_{s=1}^{m-i-1}(j_{s+1}-j_s)}=1.\no
\end{align}
Therefore, we obtain
\begin{align}
  \varlimsup_{n\rto}\frac{G_{B_i}(n,m)}{(\log n)^m}=0, 1\le i\le m-1.\no
\end{align}
We thus finish the proof of \eqref{gbi}.

Now, putting \eqref{gau}, \eqref{gal} and \eqref{gbi} together,
  we infer from  \eqref{gab} that
\begin{align}
  (\lambda_C-\ve)^m(1-\ve)^m\le \varliminf\frac{G(n,m)}{(\log n)^m}\le \varlimsup\frac{G(n,m)}{(\log n)^m}\le (\lambda_C+\ve)^m(1+\ve)^m.\no
\end{align}
Since $\ve$ is arbitrary, letting $\ve\rightarrow 0,$ we get
\begin{align}
 \lim_{n\rto} \frac{G(n,m)}{(\log n)^m}=\lambda_C^m.\no
\end{align}
Consequently, dividing by $(\log n)^k$ on both sides of \eqref{esn} and taking the limit, we conclude that
\begin{align}
  \lim_{n\rto}\frac{\mathbb E|C\cap [1,n]|^k}{(\log n)^k}=k!\lambda_C^k.\no
\end{align}
Lemma \ref{com} is proved. \qed

\subsection{Proof of Theorem \ref{cs}}
\proof We have shown in Lemma \ref{com} that for $k\ge 1,$
$\mathbb E\z(\frac{ |C(a,b)\cap [1,n]|}{2^{-a}\binom{a-1}{b-1}\log n}\y)^k\rightarrow k!$, $\mathbb E\z(\frac{|C(*,a)\cap[1,n]|}{\log n}\y)^k\rightarrow k!$ and
$\mathbb E\z(\frac{|C_w\cap [1,n]|}{2\log n}\y)^k\rightarrow k! $ as $n\rto.$  Since $ ((2n)!)^{-\frac{1}{2n}}\sim \frac{e }{2n},$ as $n\rto,$  we thus have
\begin{align*}
  \sum_{k=1}^\infty ((2k)!)^{-\frac{1}{2k}}=\infty.
\end{align*}
Consequently, using Carleman's test for the uniqueness of the moment problem (see e.g., \cite[Chap.II, \S 12]{s}), we conclude that
\begin{align}
  \frac{\z|C(a,b)\cap [1,n]\y|}{2^{-a}\binom{a-1}{b-1}\log n}\overset{ d}{\rightarrow} S,\ \frac{|C(*,a)\cap[1,n]|}{\log n}\overset{ d}{\rightarrow} S\text{ and }  \frac{\z|C_w\cap [1,n]\y|}{2\log n}\overset{ d}{\rightarrow} S\no
\end{align}
as $n\rto,$ where $S$ is  an exponential random variable with $\mathbb P(S>t)=e^{-t}, t>0.$ Theorem \ref{cs} is proved. \qed

%%%%%%%%%%%%%%%%%%%%%%%%%%%%%%%%%%%%%%%%%55

\vspace{.5cm}

\noindent{{\bf \Large Acknowledgements:}} The authors are in debt to two anonymous referees who read carefully an earlier unpublished version of our manuscript \cite{w23},  provided us with very useful suggestions and comments, and  especially helped us to improve the proofs of Lemma \ref{lems} and Lemma \ref{aaxy} in the current manuscript, respectively. We also thank Professor A. F\"oldes for confirming to us that an additional  monotonicity condition of $D(n)$ is required for \cite[Theorem 1.1]{cfrc}.   This project is partially supported by the National Natural Science Foundation of China (Grant
No. 11501008) and the Nature Science Foundation of Anhui Educational Committee (Grant No.
2023AH040025).
%%%%%%%%%%%%%%%%%%%%%%%%%%%%%%%%%%%%%%%%%%%%%%%%5

\end{document}